\documentclass{amsart}
\usepackage{amsmath}
\usepackage{amssymb}
\usepackage{amsthm}
\usepackage{amsfonts}
\usepackage{tikz-cd}
\usepackage{enumitem}
\usepackage[hidelinks]{hyperref}

\newtheorem{theorem}{Theorem}[section]
\newtheorem{proposition}[theorem]{Proposition}
\newtheorem{corollary}[theorem]{Corollary}
\newtheorem{lemma}[theorem]{Lemma}

\theoremstyle{definition}

\newtheorem{definition}[theorem]{Definition}

\theoremstyle{remark}

\newtheorem{remark}[theorem]{Remark}
\newtheorem{notation}[theorem]{Notation}

\DeclareMathOperator{\Spec}{Spec}
\DeclareMathOperator{\charac}{char}

\def\Q{\mathbb{Q}}
\def\Z{\mathbb{Z}}

\def\P{\mathbb{P}}
\def\CH{\mathrm{CH}}

\def\far{\mathrm{far}}

\def\Gm{\mathbb{G}_m}
\def\GL{\mathrm{GL}}
\def\PGL{\mathrm{PGL}}

\def\B{\mathrm{B}}
\def\PB{\mathrm{PB}}
\def\cO{\mathcal{O}}
\def\oD{\overline{\Delta}}
\def\Pic{\mathrm{Pic}}
\def\cH{\mathcal{H}}
\def\cD{\mathcal{D}}
\def\cX{\mathcal{X}}
\def\cY{\mathcal{Y}}
\def\cZ{\mathcal{Z}}
\def\cC{\mathcal{C}}
\def\cW{\mathcal{W}}

\def\cP{\mathcal{P}}
\def\cM{\mathcal{M}}

\begin{document}
	\title[THE CHOW RING OF THE STACK OF POINTED HYPERELLIPTIC CURVES]{THE INTEGRAL CHOW RING OF THE STACK OF POINTED HYPERELLIPTIC CURVES}
	\author{ALBERTO LANDI}
	\address{Department of Pure Mathematics, Brown University, 151 Thayer Street, Providence, RI 02912, USA}
	\email{alberto\_landi@brown.edu}
	
	
	\keywords{Chow ring, intersection theory, hyperelliptic curve, moduli}

	\begin{abstract}
		We study the integral Chow ring of the stack $\cH_{g,n}$ parametrizing $n$-pointed smooth hyperelliptic curves of genus $g$. We compute the integral Chow ring of $\cH_{g,n}$ for $n=1,2$ completely, while for $3\leq n\leq2g+2$ we compute it up to the additive order of a single class in degree 2. We obtain partial results also for $n=2g+3$. In particular, taking $g=2$ and recalling that $\cH_{2,n}=\cM_{2,n}$, our results hold for $\CH^*(\cM_{2,n})$ for $1\leq n\leq7$.
	\end{abstract}
	\maketitle
	

	\section*{Introduction}
	Since the seminal work of David Mumford~\cite{Mum77}, several researchers have put their efforts in computing (both integral and rational) Chow rings of natural stacks of curves, as the stack $\overline{\cM}_{g,n}$ of stable $n$-pointed curves of genus $g$, and its smooth counterpart $\cM_{g,n}$.
	
	In this paper we are interested in the closed substack $\cH_{g,n}$ of $\cM_{g,n}$ classifying $n$-pointed smooth hyperelliptic curves of genus $g$.
	Notice that for $g=2$ we have $\cM_{2,n}=\cH_{g,2}$.
	
	Since a considerable number of papers have been written on the intersection theory of stacks, we recall only a few known results about $\cH_{g,n}$ and related stacks. In~\cite{Lar19}, Eric Larson has computed the Chow ring of $\overline{\cM}_2$, while the Chow ring of $\cM_2$ was previously computed by Angelo Vistoli in~\cite{Vis98}. This last result has been generalized to $\cH_g$ by Dan Edidin and Damiano Fulghesu in~\cite{EF09} for the even genus case, and by Andrea Di Lorenzo in~\cite{DL18} for the odd genus case. The pointed case was first addressed by Michele Pernice in~\cite{Per22}, where he studied the $n=1$ case. However, the result is not completely correct, due to a mistake at the end of the proof; in this article we supplement his ideas to correct it. Finally, Dan Edidin and Zhengning Hu studied the locus of $\cH_{g,n}$ where all the sections have image in the Weierstrass divisor (in particular, $1\leq n\leq2g+2$), and they have computed the Chow rings of those loci. We will use their results in this paper. Samir Canning and Hannah Larson (\cite{CL22}) have computed the rational Chow ring of $\cH_{g,n}$ for $n\leq2g+6$, using a different description of the stack.
	
	We compute the integral Chow ring of $\cH_{g,n}$ for $n=1,2$ completely, while for $3\leq n\leq2g+3$ we find the generators and almost every relation. As a consequence, we recover~\cite[Corollary 1.5]{CL22} when $n\leq2g+3$.
	
	All our results are based on the techniques of~\cite{Lan23}, where a new description of $\cH_{g,n}$ is found and used to compute its Picard group for all $g$ and $n$.
	\subsection*{Results and Strategy}
	We work over a field of characteristic strictly greater than $2g$, where $g\geq2$ is the genus. The main result of this paper is the (almost) computation of the integral Chow ring of $\cH_{g,n}$ for $1\leq n\leq2g+3$.
	
	Before stating the main Theorems, we briefly describe some important substacks of $\cH_{g,n}$ that will appear in the statements; see the first Section for the rigorous definitions.
	
	\begin{definition}\label{def: intro classes}
		For every $1\leq i\not=j\leq n$, define $\cZ_{g,n}^{i,j}$ to be the closed substack of $\cH_{g,n}$ parametrizing pointed hyperelliptic curves whose $i$-th and $j$-th section differ by the hyperelliptic involution. For every $1\leq i\leq n$, define $\cW_{g,n}^i$ to be the closed substack of $\cH_{g,n}$ parametrizing pointed hyperelliptic curves whose $i$-th section has image in the Weierstrass locus; morally, this corresponds to the case $i=j$ that we left out in the previous definition. All these closed substacks are irreducible, smooth, of codimension 1 in $\cH_{g,n}$, and yield classes $[\cZ_{g,n}^{i,j}]$, $[\cW_{g,n}^i]\in\Pic(\cH_{g,n})$, respectively. We denote the open complement of the union of all the $\cZ_{g,n}^{i,j}$ by $\cH_{g,n}^{\far}$. Finally, we denote the restriction of the psi classes to $\cH_{g,n}$ by $\psi_1,\ldots,\psi_n$.
	\end{definition}	
	The Chow ring of $\cH_{g,n}$ has already been computed in the case $n=1$ in~\cite{Per22} by Pernice, but the result is not correct as stated; we correct it by extending his ideas. The presentation of $\mathrm{CH}^*(\cH_{g,1})$ with the most geometric meaning that we get is described in the following Theorem.
	\begin{remark}
		We use the following notation. If we have a ring $A$ and elements $\beta_1,\ldots,\beta_s\in A$, and $p_1,\ldots,p_n\in\mathbb{Z}[x_1,\ldots,x_s]$ are polynomials in $s$ variables, we write
		\[
		A\simeq\frac{\mathbb{Z}[\beta_1,\ldots,\beta_s]}{(p_1(\beta_1,\ldots,\beta_s),\ldots,p_r{(\beta_1,\ldots,\beta_s)})}
		\]
		to mean that there is a surjective homomorphism
		\[
		\begin{tikzcd}
			\mathbb{Z}[x_1,\ldots,x_s]\arrow[r,twoheadrightarrow] & A
		\end{tikzcd}
		\]
		sending $x_i$ to $\beta_i$ and with kernel generated by $p_1,\ldots,p_r$. In this setting, we also say that the $p_i(\beta_1,\ldots,\beta_s)$ generate the ideal of relations.
	\end{remark}
	\begin{theorem}\label{thm: Chow ring Hg1 intro}
		Let $\charac k>2g$. Then,
		\[
		\mathrm{CH}^*(\cH_{g,1})\simeq\frac{\mathbb{Z}[\psi_1,[\cW_{g,1}^1]]}{((4g+2)((g+1)\psi_1-(g-1)[\cW_{g,1}^1]),[\cW_{g,1}^1]^2+\psi_1\cdot[\cW_{g,1}^1])}
		\]			
		regardless of the parity of $g$.
	\end{theorem}
	We manage to extend the results to the 2-pointed case.
	\begin{theorem}\label{thm: Chow Hg2 intro}
		Let $\charac k>2g$. Then, $\mathrm{CH}^*(\cH_{g,2})$ is generated by the classes of $\cW_{g,2}^1$, $\cZ_{g,2}^{1,2}$ and $\psi_1$, and the ideal of relations is generated by
		\addtolength{\jot}{2pt}
		\begin{align*}
			\bullet & \ (4g+2)((g+1)\psi_1-(g-1)[\cW_{g,2}^1]),\\
			\bullet & \  [\cZ_{g,2}^{1,2}]^2+\psi_1\cdot[\cZ_{g,2}^{1,2}],\qquad[\cW_{g,2}^{1}]^2+\psi_1\cdot[\cW_{g,2}^1],\qquad
			[\cZ_{g,2}^{1,2}]\cdot[\cW_{g,2}^1],\\
			\bullet & \  (2g+1)(g+1)\psi_1^2+(2g+1)(3g-1)[\cW_{g,2}^1]^2-(2g+1)(g+1)[\cZ_{g,2}^{1,2}]^2.
		\end{align*}
	\end{theorem}
	
	Notice that in both cases the Chow ring is generated in degree 1. This holds for $3\leq n\leq2g+3$ as well.
	
	\begin{proposition}\label{prop: generated degree 1 intro}
		For all $3\leq n\leq2g+3$, the ring $\mathrm{CH}^*(\cH_{g,n})$ is generated by the classes $[\cZ_{g,n}^{1,2}],\ldots,[\cZ_{g,n}^{1,n}]$ and $[\cW_{g,n}^1]$.
	\end{proposition}
	
	To prove Theorem~\ref{thm: Chow Hg2 intro} and Proposition~\ref{prop: generated degree 1 intro}, we compute the Chow ring of $\cH_{g,n}^{\far}$ for $n\leq2g+3$, by means of equivariant intersection theory of~\cite{EG98}, leveraging the presentation of $\cH_{g,n}^{\far}$ as a quotient stack obtained in~\cite[Propositions 2.19 and 2.21]{Lan23}. The results are Corollary~\ref{cor: Chow Hg2far} and Proposition~\ref{prop: Chow Hgnfar n>=3}; see also Remark~\ref{rmk: Chow Hg2g+3far}. In particular, we show that $\CH^*(\cH_{g,n}^{\far})$ is generated in degree 1 for $n\leq2g+3$.
	
	Then, for every $1\leq i<j\leq n$, we identify $\cZ_{g,n}^{i,j}$ with an open substack $\cH_{g,n-1}^{i}$ of $\cH_{g,n-1}$ and use the localization sequence to get an exact sequence
	\[
	\begin{tikzcd}
		\bigoplus_{i<j}\mathrm{CH}^*(\cH_{g,n-1}^{i})\arrow[r] & \mathrm{CH}^*(\cH_{g,n})\arrow[r] & \mathrm{CH}^*(\cH_{g,n}^{\far})\arrow[r] & 0.
	\end{tikzcd}
	\]
	Using induction on $n$ and the result for $n=1$, we are able to prove Proposition~\ref{prop: generated degree 1 intro}.
	
	Moreover, for every $n$, in Section~\ref{sec: generators Chow} we find numerous relations between Chow classes using the results in~\cite{Lan23} on the Picard group of $\cH_{g,n}$, which are enough to prove Theorem~\ref{thm: Chow Hg2 intro}.
	
	\begin{remark}
	Unfortunately, the relations found are not enough to determine the integral Chow ring in the case of $3\leq n\leq2g+3$. For $3\leq n\leq2g+2$, the only missing information is the additive order of the degree 2 class $[\cZ_{g,n}^{1,2}]^2$, which we only manage to bound. The case of $2g+3$ is less well behaved. The Chow ring of $\cH_{g,2g+3}$ is very different from those for $n\leq2g+2$, as this stack is now a scheme, thus the Chow ring vanishes in degree greater than its dimension. In this case, also the multiplicative order of $[\cW_{g,n}^1]$ is unknown.
	
	See Section~\ref{sec: Hgn n>=3} for the precise statements in the case $n\geq3$, in particular Theorems~\ref{thm: Chow Hg n=3},~\ref{thm: Chow Hgn n=4},~\ref{thm: Chow Hgn 5<=n<=2g+2} and~\ref{thm: Chow Hgn n=2g+3}.
	\end{remark}
	
	\subsection*{Structure of the paper}
	In the first Section we recall some results obtained in~\cite{Lan23}, which are essential for the rest of the paper.
	In the second and third Section we compute the Chow rings of $\cH_{g,1}$ and of the open substack $\cH_{g,2}^{\far}$ of $\cH_{g,2}$, respectively. Next, in Section~\ref{sec: generators Chow}, we show that for $1\leq n\leq2g+3$ the Chow ring of $\cH_{g,n}$ is generated in degree 1 and we compute some relations. These are enough to compute the Chow ring of $\cH_{g,2}$, as shown in Section~\ref{sec: Chow Hg2}. In Section~\ref{sec: Hgn n>=3}, we compute $\CH^*(\cH_{g,n}^{\far})$ for $3\leq n\leq2g+2$, and we conclude the (almost) computation of $\CH^*(\cH_{g,n})$ for $n\leq2g+3$.
	\subsection*{Acknowledgements}
	I would like to thank Angelo Vistoli for the valuable time spent discussing ideas on this and related subjects. I am also grateful to Dan Edidin and Zhengning Hu for the useful discussions we had. I wish to warmly thank Michele Pernice for his support and taking the time to talk about this work. Finally, the author is deeply thankful to the reviewer for his suggestions that drastically improved the clarity of the exposition.
	\section{Preliminaries and the geometry of $\cH_{g,n}$}
	We will use the description of $\cH_{g,n}$ obtained in~\cite{Lan23}, where the author computed the Picard group of $\cH_{g,n}$ for all $g$ and $n$. We recall the description and results here. We will always assume that the characteristic of the base field does not divide $2g+2$.
	
	Recall that the stack $\cH_{g,n}$ is equivalent to the stack parametrizing families $(C\xrightarrow{f}P\xrightarrow{\pi}S,\widetilde{\sigma}_1,\ldots,\widetilde{\sigma}_n)$ of $n$-pointed double covers $f:C\rightarrow P$ of a Brauer-Severi scheme $\pi:P\rightarrow S$ of relative dimension 1 over $S$, branched over a divisor of degree $2g+2$ and étale over $S$. We use this interpretation throughout the paper.
	\subsection{Description of $\cH_{g,n}$}
	Let
	\[
	\begin{tikzcd}
		\cC_g\arrow[r,"f"] & \cP_g\arrow[r,"\pi"] & \cH_g
	\end{tikzcd}
	\]
	be the universal family of hyperelliptic curves, ramifying over the `universal' Brauer-Severi stack $\cP_g\rightarrow\cH_g$, and let $\alpha$ be the universal hyperelliptic involution on $\cC_g$. Let $\cC_g^n$ be the fiber product of $n$ copies of $\cC_g$ over $\cH_g$, and $\Delta_{i,j}$ be the $i,j$-th diagonal. We denote by $\alpha_i:\cC_g^n\rightarrow\cC_{g}^n$ the involution which acts as $\alpha$ on the $i$-th component and as the identity on the others.
	Notice that $\cH_{g,n}$ is the open substack of $\cC_g^n$ given by the complement of the extended diagonal. Let $\mathrm{pr}_i:\cC_g^n\rightarrow\cC_{g}$ be the projection to the $i$-th factor.
	\begin{definition}
		Let $\cW_{g,1}^1\subset\cH_{g,1}=\cC_g$ be the ramification locus of the map $\cC_g\rightarrow\cP_g$. For all $1\leq i\leq n$, define the closed substack $\cW_{g,n}^i\subset\cH_{g,n}$ as
		\[
			\cW_{g,n}^i:=\mathrm{pr}_i^{-1}(\cW_{g,1}^1)\cap\cH_{g,n}.
		\]
		For all $i\not=j$, define the closed substacks $\cZ_{g,n}^{i,j}\subset\cH_{g,n}$ as
		\[
			\cZ_{g,n}^{i,j}:=\alpha_i^{-1}(\Delta_{i,j})\cap\cH_{g,n}.
		\]
		Finally, define the open substack $\cH_{g,n}^{\far}\subset\cH_{g,n}$ as the complement of the $\cZ_{g,n}^{i,j}$'s.
	\end{definition}
	\begin{remark}
		From the definition and~\cite[Subsection 2.2]{EH21}, it follows that all $\cW_{g,n}^i$ and $\cZ_{g,n}^{i,j}$ are integral and smooth, and that $\cZ_{g,n}^{i,j}=\cZ_{g,n}^{j,i}$. For this reason, we usually assume $i<j$. Moreover, the intersection of $\cZ_{g,n}^{i,j}$ and $\cZ_{g,n}^{l,m}$ is empty if and only if $|\{i,j\}\cap\{l,m\}|=1$, see~\cite[Remark 2.2]{Lan23} and next Remark~\ref{rmk: interpretation Zij Wi}.
	\end{remark}
	The following Remark clarifies what $\cW_{g,n}^i$, $\cZ_{g,n}^{i,j}$ and $\cH_{g,n}^{\far}$ parametrize.
	\begin{remark}\label{rmk: interpretation Zij Wi}
	Denote by
	\[
	\begin{tikzcd}
		\cC_{g,n}\arrow[r,"f"] & \cP_{g,n}\arrow[r,"\pi"] & \cH_{g,n}
	\end{tikzcd}
	\]
	the universal family of pointed hyperelliptic curves over $\cH_{g,n}$, realized as a double cover of $\cP_{g,n}$.
	Consider the universal $i$-th section $\widetilde{\sigma}_i:\cH_{g,n}\rightarrow\cC_{g,n}$, which induces a section $\sigma_i=f\circ\widetilde{\sigma}_i$ of $\cP_{g,n}\rightarrow\cH_{g,n}$. Then, $\cW_{g,n}^i$ is the locus where $\widetilde{\sigma}_i$ has image in the ramification locus of $f$. Similarly, $\cZ_{g,n}^{i,j}$ is the locus where $\sigma_i=\sigma_j$, that is where $\widetilde{\sigma}_i$ and $\widetilde{\sigma}_j$ differ by the hyperelliptic involution. In particular, given an object $(C\rightarrow P\rightarrow S,\widetilde{\sigma}_1,\ldots,\widetilde{\sigma}_n)$ of $\cH_{g,n}^{\far}$, this induces an object $(P\rightarrow S,\sigma_1,\ldots\sigma_n)$ of $\cM_{0,n}$. This defines a morphism $\cH_{g,n}^{\far}\rightarrow\cM_{0,n}$, studied in~\cite[Section 4]{Lan23}.
	\end{remark}
	The closed subschemes $\cW_{g,n}^i$ and $\cZ_{g,n}^{i,j}$ define Cartier divisors on $\cH_{g,n}$, and the classes $[\cZ_{g,n}^{i,j}]$ are subject to the following relations.
	\begin{proposition}\label{prop:relationsZ}
	Let $1<i<j\leq n$. Then,
	\[
		[\cZ_{g,n}^{i,j}]=[\cZ_{g,n}^{1,i}]+[\cZ_{g,n}^{1,j}]+[\cZ_{g,n}^{2,3}]-[\cZ_{g,n}^{1,2}]-[\cZ_{g,n}^{1,3}].
	\]
	Moreover, there are no other relations between the classes $[\cZ_{g,n}^{i,j}]$.
	\end{proposition}
	\begin{proof}
		See~\cite[Corollary 2.7 and Proposition 5.2]{Lan23}.
	\end{proof}
	Now, we study the open substack $\cH_{g,n}^{\far}\subset\cH_{g,n}$. For $n\leq2g+3$, this is isomorphic to the quotient of an open subscheme of a representation of a linear algebraic group by the same group. We report here the precise statements; see~\cite[Subsections 2.2, 2.3 and 2.4]{Lan23} for details.
	\begin{proposition}[{\cite[Theorem 4.1]{AV04}}]\label{prop:presentationHg}
	Let $\mathbb{A}^{2g+3}$ be the space of binary forms of degree $2g+2$, and $\Delta_g$ the locus of singular forms. Then,
	\[
		\cH_{g}\simeq\left[\frac{\mathbb{A}^{2g+3}\setminus\Delta_g}{\mathrm{GL}_2/\mu_{g+1}}\right]
	\]
	where a matrix $A$ acts on a binary form $f$ of degree $2g+2$ as
	\[
		A\cdot f(x,y)=f(A^{-1}(x,y)).
	\]
	\end{proposition}
	\begin{notation}\label{not: notation n=1,2}
	When $n=1$, we see the space $\mathbb{A}^{2g+3}$ as parametrizing pairs $(f,s_1)$ where $f$ is a binary form of degree $2g+2$, $s_1\in\mathbb{A}^1$, and $f(0,1)=s_1^2$. When $n=2$ it parametrizes triples $(f,s_1,s_2)$ such that $f(0,1)=s_1^2$ and $f(1,0)=s_2^2$. We denote by $\Delta_{g,n}$ the locus of singular forms.
	\end{notation}
	\begin{proposition}[{\cite[Proposition 2.6]{Per22},~\cite[Corollary 2.17]{Lan23}}]\label{prop:presentationHg1Hg2far}
	We have
	\[
	\cH_{g,1}\simeq\left[\frac{\mathbb{A}^{2g+3}\setminus\Delta_{g,1}}{\B_2/\mu_{g+1}}\right]
	\]
	where a lower triangular matrix $\B_2\ni A=\begin{bmatrix}
		a & 0\\
		b & c
	\end{bmatrix}$ acts on $(f,s_1)\in\mathbb{A}^{2g+3}$ as
	\[
	A\cdot (f,s_1)=(f\circ A^{-1},c^{-(g+1)}s_1).
	\]
	
	Similarly,
	\[
	\cH_{g,2}^{\far}\simeq\left[\frac{\mathbb{A}^{2g+3}\setminus\Delta_{g,2}}{(\mathbb{G}_m\times\mathbb{G}_m)/\mu_{g+1}}\right]
	\]
	where a diagonal matrix $A$ with entries $(a,c)$ acts on $(f,s_1,s_2)\in\mathbb{A}^{2g+3}$ as
	\[
		A\cdot (f,s_1,s_2)=(f\circ A^{-1},c^{-(g+1)}s_1,a^{-(g+1)}s_2).
	\]
	\end{proposition}
	\begin{notation}\label{not: weighted projective stack}
	Given distinct integers $d_1,\ldots,d_s$ and non-negative integers $r_1,\ldots,r_n$, we denote by $\cP(d_1^{r_1},\ldots,d_s^{r_s})$ the weighted projective stack where $r_i$ variables have degree $d_i$.
	\end{notation}
	\begin{proposition}[{\cite[Proposition 2.21]{Lan23}}]\label{prop:representationweightedprojectivestack3<=n<=2g+3}
		Let $3\leq n\leq 2g+3$. Then
		\[
		\cH_{g,n}^{\far}\simeq(\cP(2^{2g+3-n},1^n)\times(\mathbb{A}^{n-3}\setminus\widetilde{\Delta}))\setminus\overline{\Delta}_{g,n}
		\]
		where $\overline{\Delta}_{g,n}$ is the locus corresponding to singular forms. In particular, for $n=2g+3$,
		\[
		\cH_{g,2g+3}^{\far}\simeq(\mathbb{P}^{2g+2}\times(\mathbb{A}^{2g}\setminus\widetilde{\Delta}))\setminus\overline{\Delta}_{g,2g+3}
		\]
		and it is a scheme.
	\end{proposition}
	In Subsection~\ref{subsec: Chow Hgnfar n>=3}, we will use these presentations to compute their Chow ring.
	\subsection{The Picard group and classes of divisors}\label{subsec:Picardgroup}
	For our computation it will be crucial to know the Picard group of $\cH_{g,n}$, which has been computed in~\cite{Lan23} in a more general setting. We report the results here.
	\begin{theorem}[{\cite[Corollary 5.14]{Lan23}}]\label{thm: abstract Picard}
	Suppose that the ground field $k$ is of characteristic not dividing $2g+2$.
	\begin{itemize}
		\item If $g$ is odd, then \/ $\Pic(\cH_{g,n})\simeq\mathbb{Z}^{n}\oplus\mathbb{Z}/(8g+4)\mathbb{Z}$.
		\item If $g$ is even, then \/ $\Pic(\cH_{g,n})\simeq\mathbb{Z}^{n}\oplus\mathbb{Z}/(4g+2)\mathbb{Z}$.
	\end{itemize}
	Moreover, the torsion part always comes from $\cH_{g}$.
	\end{theorem}
	For $n\geq3$, it is easy to describe the generators of the Picard group.
	\begin{proposition}[{\cite[Proposition 5.16]{Lan23}}]\label{prop: generators Picard}
		Let $n\geq3$ and assume that the characteristic of $k$ does not divide $2g+2$. Then
		\[
		\{[\cZ_{g,n}^{1,j}]\}_{2\leq j\leq n} \cup\{[\cZ_{g,n}^{2,3}]\}\cup\{[\cW_{g,n}^1]\}
		\]
		is a minimal set of generators of $\Pic(\cH_{g,n})$. Moreover,
		\begin{equation}
			(8g+4)[\cW_{g,n}^1]=(4g+2)(g+1)([\cZ_{g,n}^{1,2}]+[\cZ_{g,n}^{1,3}]-[\cZ_{g,n}^{2,3}])
		\end{equation}
		is the only relation, regardless of the parity of $g$.
	\end{proposition}
	To address the case of $n\leq2$ it is important to compute the classes of $\cW_{g,n}^1$ and $\psi_1$ in terms of a fixed base of $\Pic(\cH_{g,1})$. To choose one such base, we use the conventions of~\cite[Theorem 3.2]{Lan23}, which are summarized in the following Remark.
	
	\begin{remark}\label{rmk: explicit character basis Picard}
	For the proof of the next statements, see~\cite[Subsection 3.1.1]{Lan23}. For $n=1,2$, using the presentations in Proposition~\ref{prop:presentationHg1Hg2far}, the Picard group of $\cH_{g,1}$ and $\cH_{g,2}^{\far}$ can be seen as the quotient of the group of characters
	\[
		(\B_2/\mu_{g+1})^{\widehat{}}\simeq(T/\mu_{g+1})^{\widehat{}}\subset\widehat{T}\simeq\Z\oplus\Z
	\]
	by the subgroup generated by the class of $\Delta_{g,n}$, which is $2(g+1)(2g+1)\cdot(1,1)$. Here, $T\simeq\Gm\times\Gm$ is the maximal torus of $\B_2$ consisting of diagonal matrices, and the last isomorphism is induced by the two natural projections $T\rightarrow\Gm$. Here, the subgroup $(T/\mu_{g+1})^{\widehat{}}$ of $\widehat{T}$ is the set of pairs $(a,b)$ with $a+b$ divisible by $g+1$. If $g$ is odd, the isomorphism of Theorem~\ref{thm: abstract Picard} is induced by $(a,b)\mapsto((a-b)/2,(a+b)/(g+1))$, while if $g$ is even it is induced by $(a,b)\mapsto(a-b,(-ga+(g+2)b)/2(g+1))$. We will use these explicit bases for the following results. Notice that with these bases the pullback $\Pic(\cH_{g,1})\rightarrow\Pic(\cH_{g,2}^{\far})$ along the map $\cH_{g,2}^{\far}\rightarrow\cH_{g,1}$ that forgets the last section is the identity.
	\end{remark}
	
	\begin{lemma}[{\cite[Lemma 5.7]{Lan23}}]\label{lem:classW}
	Using the bases described in Remark~\ref{rmk: explicit character basis Picard},
	\[
		[\cW_{g,1}^1]=
		\begin{cases}
			(-(g+1)/2,1)\in\mathbb{Z}\oplus\mathbb{Z}/(8g+4)\mathbb{Z} & \text{ if }g\text{ odd},\\
			(-(g+1),g/2+1)\in\mathbb{Z}\oplus\mathbb{Z}/(4g+2)\mathbb{Z} & \text{ if }g\text{ even},
		\end{cases}
	\]
	in $\Pic(\cH_{g,1})$ and $\Pic(\cH_{g,2}^{\far})$.
	Similarly, in $\Pic(\cH_{g,2}^{\far})$ we have
	\[
		[\cW_{g,2}^2]=
		\begin{cases}
			((g+1)/2,1)\in\mathbb{Z}\oplus\mathbb{Z}/(8g+4)\mathbb{Z} & \text{ if }g\text{ odd},\\
			(g+1,-g/2)\in\mathbb{Z}\oplus\mathbb{Z}/(4g+2)\mathbb{Z} & \text{ if }g\text{ even},.
		\end{cases}
	\]
	\end{lemma}
	\begin{proposition}\label{prop:classpsi1}
		Using the basis described in Remark~\ref{rmk: explicit character basis Picard}, in $\Pic(\cH_{g,1})$ we have
		\[
			\psi_1=
			\begin{cases}
				(-(g-1)/2,1)\in\mathbb{Z}\oplus\mathbb{Z}/(8g+4)\mathbb{Z} & \text{ if }g\text{ odd},\\
				(-(g-1),g/2)\in\mathbb{Z}\oplus\mathbb{Z}/(4g+2)\mathbb{Z} & \text{ if }g\text{ even}.
			\end{cases}
		\]
	\end{proposition}
	\begin{proof}
		Consider the linear group $\B_2\times\mathbb{G}_m$ and its action on $\mathbb{A}^{2g+3}\times\mathbb{A}^3$ described as follows. Let $B\in\B_2$ and $\lambda\in\mathbb{G}_m$, and interpret the space $\mathbb{A}^{2g+3}\times\mathbb{A}^3$ as parametrizing tuples $((f,s_1),(x_0,x_1,y))$ with $f$ a binary form of degree $2g+2$ satisfying $f(0,1)=s_1^2$. Then the action is defined by the rule
		\[
		(B,\lambda)\cdot((f,s_1),(x_0,x_1,y))=((\lambda^{-2}f\circ B^{-1},\lambda^{-1}c^{-(g+1)}s_1),(B(x_0,x_1),\lambda^{-1}y))
		\]
		where
		\[B=\begin{bmatrix}
			a & 0\\
			b & c
		\end{bmatrix}\]
		and $B(x_0,x_1)=(ax_0,bx_0+cx_1)$.
		Notice that we have homomorphisms
		\begin{equation}\label{eq: homomorphism groups universal family Hg1}
		\begin{tikzcd}
			\B_2/\mu_{g+1}\arrow[r] & \B_2\times\mathbb{G}_m\arrow[r] & \B_2/\mu_{g+1}
		\end{tikzcd}
		\end{equation}
		sending $[A]=\begin{bmatrix}
			a & 0\\
			b & c
		\end{bmatrix}$ to $(c^{-1}A,c^{g+1})$ and sending $(B,\lambda)$ to the class of $\lambda^{1/(g+1)}B$, which are easily seen to be well defined. Notice that the composite is just the identity. Moreover, the projection
		\[
		\begin{tikzcd}
			\mathbb{A}^{2g+3}\times\mathbb{A}^3\arrow[r] & \mathbb{A}^{2g+3}, & ((f,s_1),(x_0,x_1,y))\mapsto(f,s_1)
		\end{tikzcd}
		\]
		admits a section
		\begin{equation}\label{eq: universal section Hg1}
		\begin{tikzcd}
			\mathbb{A}^{2g+3}\arrow[r] & \mathbb{A}^{2g+3}\times\mathbb{A}^3, & (f,s_1)\mapsto((f,s_1),(0,1,s_1)),
		\end{tikzcd}
		\end{equation}
		and both maps are equivariant with respect to the action of the two linear groups $\B_2\times\Gm$ and $\B_2/\mu_{g+1}$ and the homomorphisms in~\eqref{eq: homomorphism groups universal family Hg1}. Moreover, this last section has image in the closed subscheme $X'$ of $\mathbb{A}^{2g+3}\times\mathbb{A}^3$ defined by the equation $f(x_0,x_1)=y^2$. Notice that this locus is $\B_2\times\mathbb{G}_m$-invariant. Excising the locus of singular polynomials, we get an open invariant subscheme $X$ of $X'$. It is not hard to show that the induced morphism
		\[
		\begin{tikzcd}
			\varphi:\left[\frac{X}{\B_2\times\mathbb{G}_m}\right]\arrow[r] & \cH_{g,1}
		\end{tikzcd}
		\]
		is the universal family over $\cH_{g,1}$, with universal section $\widetilde{\sigma}$ induced by~\eqref{eq: universal section Hg1}. We can use this construction to compute the class of $\psi_1$ in $\Pic(\cH_{g,1})$.
		
		Notice that, using the exact sequence of differentials, the module of relative differentials of the morphism $\varphi$ seen at the level of schemes, meaning without passing to the quotient, is given by
		\[
		\frac{dx_0\oplus dx_1\oplus dy}{\langle f_{x_0}dx_0+f_{x_1}dx_1-2ydy\rangle}.
		\]
		Write a binary form $f$ of degree $2g+2$ as $f(x_0,x_1)=\sum_im_ix_0^{2g+2-i}x_1^i$, thus $m_0,\ldots,m_{2g+1},s_1$ give coordinates on $\mathbb{A}^{2g+3}$. Then, restricting along the section $\widetilde{\sigma}$, we get
		\[
		\frac{dx_0\oplus dx_1\oplus dy}{\langle m_{2g+1}dx_0+(2g+2)s_1^2dx_1-2s_1dy\rangle}
		\]
		which is a rank 2 vector bundle over $\mathbb{A}^{2g+3}\setminus\Delta_{g,1}$. Take the determinant, which is generated by $dx_0\wedge dy$ over $s_1\not=0$, and by $dx_1\wedge dy$ over $m_{2g+1}\not=0$. Notice that, when both $s_1$ and $m_{2g+1}$ are non-zero,
		\[
		\frac{1}{s_1^2}dx_0\wedge dy=\frac{2g+2}{m_{2g+2}}dx_1\wedge dy.
		\]
		This defines a non-vanishing regular section (recall that the characteristic of the base field does not divide $2g+2$), hence a trivialization of the pullback of the dual sheaf. Now, $(dx_0\wedge dy)/s_1^2$ is an invariant section generating the line bundle, and the action of $\B_2/\mu_{g+1}$ on it is via $a^{-1}c^{-g}$. We conclude by applying the change of basis as in Remark~\ref{rmk: explicit character basis Picard}.
	\end{proof}
	Together, the two results above yield the following Proposition.
	\begin{proposition}\label{prop: generators Picard n=1,2}
		The classes $\psi_1$ and $[\cW_{g,1}^1]$ form a minimal set of generators of $\Pic(\cH_{g,1})$, and
		\[
		(4g+2)(g+1)\psi_1=(4g+2)(g-1)[\cW_{g,1}^1]
		\]
		is the only relation. Adjoining $[\cZ_{g,2}^{1,2}]$, they give a minimal set of generators of $\Pic(\cH_{g,2})$, with the same relation.
	\end{proposition}
	We conclude the Section by describing the pullback $p^*:\Pic(\cH_{g,2}^{\far})\rightarrow\Pic(\cH_{g,3}^{\far})$ along the morphism $p:\cH_{g,3}^{\far}\rightarrow\cH_{g,2}^{\far}$ that forgets the last section, useful in Section~\ref{sec: generators Chow}.
	\begin{lemma}\label{lem: pullback Hg3far Hg2far}
		Using the bases of Remark~\ref{rmk: explicit character basis Picard} for $\Pic(\cH_{g,2}^{\far})$, the pullback along the morphism $p:\cH_{g,3}^{\far}\rightarrow\cH_{g,2}^{\far}$ that forgets the last section is
		\begin{equation}
			\Pic(\cH_{g,2}^{\far})\ni(a,b)\mapsto\begin{cases}
				b & \in\Z/(8g+4)\Z\simeq\Pic(\cH_{g,3}^{\far}),\quad \text{ if }g\text{ is odd},\\
				a+2b & \in\Z/(4g+2)\Z\simeq\Pic(\cH_{g,3}^{\far}),\quad \text{ if }g\text{ is even}.
			\end{cases}
		\end{equation}
		In particular,
		\begin{equation}\label{eq: equality W1 W2 psi1 on Hg3far}
			p^*[\cW_{g,2}^1]=p^*[\cW_{g,2}^2]=p^*\psi_1.
		\end{equation}
	\end{lemma}
	\begin{proof}
		The first part follows from the explicit computations of~\cite[Section 3]{Lan23}. Equation~\ref{eq: equality W1 W2 psi1 on Hg3far} follows from the first part, Lemma~\ref{lem:classW} and Proposition~\ref{prop:classpsi1}.
	\end{proof}
	\section{The Chow ring of $\cH_{g,1}$}\label{sec:ChowringHg1}
	In this Section we extend and correct the ideas of Pernice in~\cite{Per22} to compute the Chow ring of $\cH_{g,1}$. The only difference in the result is the relation in degree 2, which essentially amounts to `half' of the one of Pernice in the even genus case, and to `one fourth' in the odd genus case.
	
	Recall that $\cH_{g,n}\simeq[(\mathbb{A}^{2g+3}\setminus\Delta_{g,1})/(\B_2/\mu_{g+1})]$, by Proposition~\ref{prop:presentationHg1Hg2far}. By the localization sequence, we need to know the Chow ring of the classifying stack of $\B_2/\mu_{g+1}$ and compute the image of the pushforward along $i:\Delta_{g,1}\hookrightarrow\mathbb{A}^{2g+3}$.
	
	The first is well known, thanks to the isomorphisms
	\begin{equation}\label{eq: isom B2}
		\B_2/\mu_{g+1}\simeq
		\begin{cases}
			\B_2 & \text{ if }g\text{ is even},\\
			\Gm\times\PB_2 & \text{ if }g\text{ is odd},
		\end{cases}
	\end{equation}
	see~\cite[Proposition 2.7]{Per22}. In both cases, $\B_2/\mu_{g+1}$ is a unipotent split extension of a split 2-dimensional torus $T$, hence one can reduce to working with its maximal torus $T=(\Gm\times\Gm)/\mu_{g+1}$ consisting of diagonal matrices. Under the isomorphisms in~\eqref{eq: isom B2}, $T$ is mapped to the natural maximal torus $\Gm\times\Gm$ of the two groups on the right, depending on the parity of $g$. From now on, we will work $T$-equivariantly under the identifications induced by~\eqref{eq: isom B2}.
	
	To compute the image of $i_*$, we follow Pernice and use the theory of Chow envelopes, that we briefly recall now. For the proofs and more details see~\cite{EF09} and~\cite{Per22}; for the general theory on equivariant Chow rings, see~\cite{EG98}.
	\begin{definition}\label{def: Chow envelope Chow surjective}
		Let $f:\cX\rightarrow\cY$ be a proper, representable morphism of quotient stacks. We say that $f$ is Chow-surjective if the morphism of groups
		\[
		\begin{tikzcd}
			f_*:\mathrm{CH}_*(X)\arrow[r] & \mathrm{CH}_*(Y)
		\end{tikzcd}
		\]
		is surjective.
		
		We say that a morphism of algebraic stacks $f:\cX\rightarrow\cY$ is a Chow envelope if $f(K):\cX(K)\rightarrow\cY(K)$ is essentially surjective for every extension of fields $K/k$.
	\end{definition}
	\begin{proposition}[{\cite[Corollary 4.6]{Per22}}]\label{prop: Chow envelope surjectivity with T}
		Let $G$, $T$ be two groups schemes over $k$ and suppose we have an action of $G\times T$ on two algebraic spaces $X$ and $Y$ and a proper map $f:X\rightarrow Y$ which is $G\times T$-equivariant. Assume that $T$ is a special group. Suppose that $f^G$ is a Chow envelope and $f^{G\times T}$ is a proper representable morphism. Then $f^{G\times T}:[X/(G\times T)]\rightarrow[Y/(G\times T)]$ is Chow-surjective.
	\end{proposition}
	We set $N=2g+2$. Since we want to use the theory of Chow envelopes, we pass to the projective setting. To do so, we form the quotient $\cP(2^N,1)$ of $\mathbb{A}^{N+1}\setminus0$ by the action of $\mathbb{G}_m$ with weights $(2,\ldots,2,1)$, which carries an induced $T$-action, and let $\oD_{g,1}$ be the image of $\Delta_{g,1}$ under the quotient. After computing the image of the pushforward along the inclusion $i:\oD_{g,1}\hookrightarrow\cP(2^N,1)$, we can recover the Chow ring of $\cH_{g,1}$ using the $\Gm$-bundle formula, see~\cite[Proposition 3.3]{Per22}.
	
	\begin{remark}\label{rmk: interpretation Projective stack and tautological}
	We interpret a weighted projective stack $\cP(2^r,1)\simeq[(\mathbb{A}^{r+1}\setminus0)/\Gm]$ as parametrizing pairs $(f,s_1)$, where $f$ is a binary form of degree $r$ satisfying $f(0,1)=s_1^2$, and $(f,s_1)$ is equivalent to $(g,t_1)$ if there exists $\lambda\in\Gm$ such that $g=\lambda^2f$ and $t_1=\lambda s_1$. We interpret $\cP(2^r)$ as parametrizing binary forms $f$ of degree $r-1$, where $f$ and $g$ are equivalent if there exists $\lambda\in\Gm$ such that $g=\lambda^2f$.
	Finally, every projective space $\P^r$ is seen as parametrizing binary forms of degree $r$, up to standard scaling.
	\end{remark}
	
	In~\cite[Proposition 4.12]{Per22}, the author proves that the maps
	\[
	\begin{tikzcd}[row sep=tiny]
		a_r:\mathbb{P}^r\times\cP(2^{N-2r},1)\arrow[r] & \cP(2^N,1), & (f,(g,s_1))\mapsto(f^2g,f(0,1)s_1)\\
		b_r:\mathbb{P}^{r-1}\times\cP(2^{N-2r})\arrow[r] & \cP(2^N,1), & (f,g)\mapsto(x^2f^2g,0)
	\end{tikzcd}
	\]
	are $T$-equivariant with respect to some $T$-actions described in~\cite[Remark 4.7]{Per22}, and factor through $\oD_{g,1}$. Moreover, by taking their disjoint union with $1\leq r\leq N/2$ varying, they form a Chow envelope of $\oD_{g,1}$. In particular, by Proposition~\ref{prop: Chow envelope surjectivity with T}, it is enough to compute $a_{r*}$ and $b_{r*}$.
	
	The $T$-equivariant Chow ring of the sources of $a_r$ and $b_r$ are computed using the projection bundle formula; in particular, as $\CH(BT)$-algebras, they are generated by the pullback of the first Chern class of the tautological line bundles on the two components; see~\cite[Remark 5.1]{Per22}.
	\begin{notation}\label{not: notation xi}
	For every $j\geq1$, set $\xi_j:=c_1^T(\cO_{\P^j}(1))$, and $\xi_0:=1\in\CH_T^*(\Spec(k))$.
	\end{notation}
	Using the projection formula, we get the following:
	\begin{lemma}
		The image of $i_*$ is generated by the classes $\alpha_{r,i}:=a_{r*}(\xi_r^i)$ for $0\leq i\leq r$, and $\beta_{r,i}:=b_{r*}(\xi_{r-1}^i)$ for $0\leq i\leq r-1$.
	\end{lemma}
	
	Therefore, we are left with computing the ideal generated by the classes $\alpha_{r,i}$ and $\beta_{r,i}$. Before doing so, we explain why the method in~\cite{Per22} does not work.
	\subsection{Problems with the previous proof}\label{subsec:problems}
	The main problem of the computation in~\cite{Per22} is the following. Recall that $\cH_g\simeq[(\mathbb{A}^{N+1}\setminus\Delta_g)/(\GL_2/\mu_{g+1})]$, and
	\begin{equation}\label{eq: isom GL2}
		\GL_2/\mu_{g+1}\simeq
		\begin{cases}
			\GL_2 & \text{ if }g\text{ is even},\\
			\Gm\times\PGL_2 & \text{ if }g\text{ is odd},
		\end{cases}
	\end{equation}
	see~\cite[Proposition 4.4]{AV04}; restricting to $\B_2/\mu_{g+1}$ one gets the isomorphisms in~\eqref{eq: isom B2}. The Chow ring of $\cH_g$ has been computed in~\cite{EF09} and~\cite{DL18} with a similar technique: they pass to the projective setting, compute the image of the pushforward along $\oD_g\hookrightarrow\P^{N+1}$, and then use the $\Gm$-bundle formula to recover the Chow ring of $\cH_g$. To do so, they consider the equivariant Chow envelope consisting in the disjoint union of maps $\pi_r:\P^r\times\P^{N-2r}\rightarrow\P^N$, $(f,g)\mapsto f^2g$.
	
	The idea in~\cite{Per22} is then to leverage the known results on $\cH_g$ to compute $a_{r*}$ and $b_{r*}$. Pernice forms $T$-equivariant cartesian diagrams
	\[
	\begin{tikzcd}
		X\arrow[r]\arrow[d,"\psi"] & \cP(2^N,1)\arrow[d,"\phi"]\\
		\mathbb{P}^r\times\mathbb{P}^{N-2r}\arrow[r,"\pi_r"] & \mathbb{P}^N
	\end{tikzcd}
	\]
	and
	\begin{equation}\label{eq: diag br}
	\begin{tikzcd}
		Y\arrow[r]\arrow[d,"\psi'"] & \cP(2^N,1)\arrow[d,"\phi"]\\
		\mathbb{P}^{r-1}\times\mathbb{P}^{N-2r}\arrow[r,"\pi_r'"] & \mathbb{P}^N
	\end{tikzcd}
	\end{equation}
	where $\pi'_r(f,g)=x^2f^2g$, and $\phi(f,s_1)=f$.
	There is an induced map $a'_r$ from $\mathbb{P}^r\times\cP(2^{N-2r},1)$ to $X$, which factors $a_r$, and a map $b'_r$ from $\mathbb{P}^{r-1}\times\cP(2^{N-2r})$ to $Y$, which factors $b_r$. These maps are again proper and representable, see~\cite[Proposition 5.4]{Per22}, where he also shows that $a'_{r*}(1)=1$, $b'_{r*}(1)=1$. Then, he uses compatibility of proper pushforward and flat pullback of~\cite{Vis98} to conclude that everything comes from $\phi^*$, if 1 is in the image of the pullback under $\psi'$. While it is true that $\phi$ is flat, see Proposition~\ref{prop:rootstack}, the pullback of 1 under $\psi'$ is not 1. Indeed, for $r=1$, we will see that $Y$ is irreducible but not smooth, not even generically reduced. Precisely, the fundamental class $[Y]$ in $\mathrm{CH}_*(Y)$ corresponds to 2. This prevents the argument of Pernice to go through.
	
	Moreover, notice that $\phi$ behaves as the root stack $\sqrt[2]{(\cO(1),m_{N})/\mathbb{P}^N}\rightarrow\mathbb{P}^N$, where $m_0,\ldots,m_{N}$ are coordinates on $\P^N$ such that every binary form $f\in\P^N$ is written as $f(x_0,x_1)=\sum_i m_ix^{N-i}y^i$.
	We show that they are indeed equal; this proves that $\phi$ is flat. The result is slightly more general.
	\begin{notation}\label{notation: hyperplanes n=1}
		Given $f\in P^N$, we denote by $m_0,\ldots,m_N$ its coefficients, so that $f(x,y)=\sum_{i}m_ix^{N-i}y^i$. This induces coordinates on $\P^N$, that we denote by the same symbols. Similarly, if we interpret $\cP(r^N,1)$ as the space of pairs $(f,s_1)$ with $f$ binary form of degree $N$ such that $f(0,1)=s_1^r$, we define coordinates $m_0,\ldots,m_{N-1},s_1$ on $\cP(r^N,1)$, so that every element $(f,s_1)$ has $f$ of the form above, with $m_N=s_1^r$. Usually, we take $r=2$.
	\end{notation}
	\begin{proposition}\label{prop:rootstack}
		Let $N$ and $r$ be positive integers. The morphism
		\[
		\begin{tikzcd}
			\phi:\cP(r^N,1)\arrow[r] & \mathbb{P}^N, & (f,s_1)\mapsto f
		\end{tikzcd}
		\]
		identifies $\cP(r^N,1)$ with the $r$-th root stack $\sqrt[r]{(\cO(1),m_{N})/\mathbb{P}^N}$ over $\mathbb{P}^N$. In particular, $\phi$ is flat.
	\end{proposition}
	\begin{proof}
		The last coordinate $s_1$ of $\cP(r^N,1)$ gives a section of $\cO_{\cP(r^N,1)}(1)$, which is an $r$-th square root of $\phi^*(\cO_{\P^N}(1),m_N)$. Therefore, $\phi$ factors as
		\[
		\begin{tikzcd}
			\cP(r^N,1)\arrow[r,"\psi"] & \sqrt[r]{(\cO(1),m_{N})/\mathbb{P}^N}\arrow[r] & \mathbb{P}^N
		\end{tikzcd}
		\]
		and $\psi$ is an isomorphism over the locus where $m_{N}\not=0$, as $\phi$ and the last map are. Now, $\psi$ induces an isomorphism between the automorphism groups of geometric points, hence it is representable. Being also proper and quasi-finite, $\psi$ is finite. As it is also a birational morphism of smooth stacks, $\psi$ is an isomorphism.
		The flatness of $\phi$ follows from~\cite[Corollary 2.3.6]{Cad04}.
	\end{proof}
	\begin{corollary}
		There is a commutative diagram with cartesian squares
		\[
		\begin{tikzcd}[column sep=large]
			\cP(2^{N-1})\arrow[r,hookrightarrow] & \cP(2^{N-1})_\epsilon\arrow[r,hookrightarrow]\arrow[d] & \cP(2^N)_\epsilon\arrow[r,hookrightarrow,"m_N=s_1^2=0"]\arrow[d] & \cP(2^N,1)\arrow[d]\arrow[dd,bend left=60,"\phi"]\\
			& \cP(2^{N-1})\arrow[r,hookrightarrow]\arrow[d] & \cP(2^N)\arrow[r,hookrightarrow]\arrow[d] & \cP(2^{N+1})\arrow[d]\\
			& \mathbb{P}^{N-2}\arrow[r,hookrightarrow,"m_{N-1}=0"] & \mathbb{P}^{N-1}\arrow[r,hookrightarrow,"m_N=0"] & \mathbb{P}^N.
		\end{tikzcd}
		\]
		such that the composite $\cP(2^{N-1})\rightarrow\cP(2^N,1)$ is $b_1$, and $\cP(2^{N})_\epsilon$ \emph{(}respectively, $\cP(2^{N-1})_\epsilon$\emph{)} is a first order thickening of $\cP(2^N)$ \emph{(}respectively, of $\cP(2^{N-1})$\emph{)}.
		
		In particular,
		\[
		{[\cP(2^{N-1})_\epsilon]}=2\in\mathrm{CH}_*(\cP(2^{N-1})_{\epsilon}).
		\]
	\end{corollary}
	\begin{proof}
		The only thing to be proved is that $\cP(2^N)_\epsilon$ is a first order thickening of $\cP(2^N)$. Notice that $\cP(2^N)_\epsilon\rightarrow\mathbb{P}^{N-1}$ is the locus of the root stack $\cP(2^N,1)\rightarrow\P^N$ where $m_N$ vanishes. Recall that when an $n$-root stack over an algebraic stack $\cX$ is defined by $(\cO_{\cX},f)$, where $f$ is a global section of $\cX$, the root stack is isomorphic to the quotient stack of
		\[
		\Spec_{\cX}(\cO_{\cX}[T]/(T^n-f))
		\]
		by the action of $\mu_n$ given by $\zeta\cdot T=\zeta T$. In particular, when $f=0$ and $n=2$, we get
		\[
		\left[\left(\Spec_{\cX}(\cO_{\cX}[T]/(T^2))\right)/\mu_2\right]
		\]
		which is a first order thickening of the root gerbe $\sqrt[2]{\cO_{\cX}/\cX}$; see~\cite[Chapter 10.3]{Ols16} or~\cite[Section 2]{Cad04} for more on root stacks. In particular, $\cP(2^N)_\epsilon$ is Zariski-locally on $\mathbb{P}^{N-1}$ of the form described above.
	\end{proof}
	\begin{corollary}\label{cor:genericisomar}
		Consider the commutative diagram with cartesian square
		\[
		\begin{tikzcd}
			\mathbb{P}^r\times\cP(2^{N-2r},1)\arrow[rd,"\theta"']\arrow[rr,bend left=20,"a_r"]\arrow[r,"a_r'"] & X\arrow[r]\arrow[d,"\psi"] & \cP(2^N,1)\arrow[d,"\phi"]\\
			& \mathbb{P}^r\times\mathbb{P}^{N-2r}\arrow[r,"\pi_r"] & \mathbb{P}^N.
		\end{tikzcd}
		\]
		Then $\phi$, $\theta$, $a_r'$ and $\psi$ are isomorphisms over the open locus where $m_N\not=0$ \emph{(}hence $s_1\not=0$\emph{)}, with all spaces isomorphic to $\mathbb{A}^r\times\mathbb{A}^{N-2r}$.
	\end{corollary}
	Now that we have a better global picture, we are ready to compute the classes $\alpha_{r,i}$ and $\beta_{r,i}$.
	\subsection{Computation of $\alpha_{r,i}$}
	First we focus on $\alpha_{r,0}$.
	\begin{lemma}\label{lem:alphar0}
		The classes $\alpha_{r,0}$ are contained in the image of $\phi^*\circ\pi_{r*}$.
	\end{lemma}
	\begin{proof}
		It follows from Corollary~\ref{cor:genericisomar} that $a'_{r*}(1)=\psi^*(1)=1$, and we conclude by compatibility of representable proper pushforward and flat pullback.
	\end{proof}
	Now we consider $\alpha_{r,i}$ for $0<i\leq r$. We have a commutative diagram
	\[
	\begin{tikzcd}
		\mathbb{P}^{r-1}\times\cP(2^{N-2r},1)\arrow[rd,"\theta'"']\arrow[r,"j"] & \mathbb{P}^r\times\cP(2^{N-2r},1)\arrow[r,"a_r"] & \cP(2^N,1)\\
		& \mathbb{P}^{r-1}\times\cP(2^{N-2r+1})\arrow[ur,"b_r"']
	\end{tikzcd}
	\]
	where $j$ is the inclusion of the hyperplane where the first coordinate of $\mathbb{P}^r$ is 0, and $\theta'(f,(g,s_1))=(f,g)$. Notice that the $T$-equivariant class of that hyperplane is the sum of $\xi_r$ and the first Chern class of a character of $T$, thanks to Lemma 2.4 of~\cite{EF09}. This shows that $\alpha_{r,1}$ is in the ideal generated by $\alpha_{r,0}$ and the image of $b_{r-1*}$. We inductively obtain the same result for all $\alpha_{r,i}$, proving the following Lemma.
	\begin{lemma}\label{lem:alpha>=1}
		For $0<i\leq r$, the class $\alpha_{r,i}$ is contained in the ideal generated by $\alpha_{r,0}$ and the image of $b_{r-1*}$.
	\end{lemma}
	We are left with computing $\beta_{r,i}$'s.
	\subsection{Computation of $\beta_{r,i}$}
	For all $1<r\leq N/2$, consider the commutative diagram
	\[
	\begin{tikzcd}
		\mathbb{P}^{r-1}\times\cP(2^{N-2r+1})\arrow[d,"b_r'"']\arrow[r,"b_r"] & \cP(2^N,1)\\
		\cP(2^{N-1})\arrow[ru,"b_1"']
	\end{tikzcd}
	\]
	where $b_r'$ sends $(f,g)$ to $f^2g$. Therefore, it is enough to compute the image of $b_{1*}$.
	
	Notice that
	\[
	\begin{tikzcd}
		b_1^*:\mathrm{CH}_T^*(\cP(2^N,1))\arrow[r] & \mathrm{CH}_T^*(\cP(2^{N-1}))
	\end{tikzcd}
	\]
	is surjective, as both Chow rings are generated by $c_1^{T}(\cO(1))$ as $\CH(BT)$-algebras. By the projection formula it follows that the image of $b_{1*}$ is the ideal generated by $b_{1*}(1)=\beta_{1,0}$. We are left with computing $\beta_{1,0}$.
	\begin{lemma}
		We have
		\[
		\beta_{1,0}=[(m_{N-1}=0)]\cdot[(s_1=0)]\in\mathrm{CH}_T^2(\cP(2^N,1)).
		\]
	\end{lemma}
	\begin{proof}
		It holds essentially by definition.
	\end{proof}
	\begin{remark}\label{rmk:projectivebundlerelationHg1}
		In particular, in the setting of Subsection~\ref{subsec:problems}, this shows that indeed $\psi'^*(1)\not=1$, and diagram~\eqref{eq: diag br} together with $\phi^*([m_N=0])=2[s_1=0]$ gives
		\[
			\phi^*(\pi'_{1*}(1))=\phi^*([(m_N=0)]\cdot[(m_{N-1}=0)])=2\beta_{1,0}.
		\]
	\end{remark}
	\subsection{Final computation}
	Putting everything together, we have proved that the relations are generated by those coming from $\cH_g$ and
	\[
	\beta_{1,0}=[(m_{N-1}=0)]\cdot[(s_1=0)]\in\mathrm{CH}_T^2(\cP(2^N,1)),
	\]
	together with the relation in $\mathrm{CH}^*_T(\cP(2^N,1))$ coming from the projective bundle formula. However, this last relation is superfluous, as it is given by the product of the $T$-equivariant classes of the $N+1$ coordinate hyperplanes, and $\beta_{1,0}$ is the product of two of those.
	
	By computing $\beta_{1,0}$ using~\cite[Lemma 2.4]{EF09} we get the following Theorem, where we use a similar notation as~\cite[Corollary 1.3]{Per22} to facilitate the comparison. We identify $T$ with $\Gm\times\Gm$ via the isomorphisms of~\eqref{eq: isom B2}, which depend on the parity of $g$; in both cases, we denote by $t_1$, $t_2$ the first Chern class of the two characters given by the two projections $T\rightarrow\Gm$.
	\begin{theorem}\label{thm:ChowHg1}
		Let $\charac k>2g$. If $g$ is even, then
		\[
		\mathrm{CH}^*(\cH_{g,1})\simeq\frac{\mathbb{Z}[t_1,t_2]}{((4g+2)(t_1+t_2),\frac{g(g-1)}{2}(t_1^2+t_2^2)-g(g+3)t_1t_2+(2g+1)(t_1+t_2)t_2)}.
		\]
		If $g$ is odd, then
		\[
		\mathrm{CH}^*(\cH_{g,1})\simeq\frac{\mathbb{Z}[t_1,t_2]}{((8g+4)t_1,2t_1^2+\frac{g(g+1)}{2}t_2^2-(2g+1)t_1t_2)}.
		\]
	\end{theorem}
	\begin{remark}
		Notice that, for $g$ even, doubling the degree 2 relation and using the degree 1 relation we get back the degree 2 relation described in~\cite{Per22}. When $g$ is odd, we need to multiply by 4.
	\end{remark}
	\subsection{Change of basis}
	We know that $[\cW_{g,1}^1]$ and $\psi_1$ generate the Chow ring, so we write those classes in terms of the generators $t_1$ and $t_2$, so that we can rewrite the above isomorphisms in terms of the classes we are interested in.
	\begin{lemma}\label{lem:conversionbasisHg1}
		With the notation of Theorem~\ref{thm:ChowHg1}, we have the following. If $g$ is even, then
		\begin{align*}
			{[\cW_{g,1}^1]}=-\frac{g}{2}t_1+\left(\frac{g}{2}+1\right)t_2,\qquad\psi_1=-\left(\frac{g}{2}-1\right)t_1+\frac{g}{2}t_2.
		\end{align*}
		If $g$ is odd, then
		\begin{align*}
			{[\cW_{g,1}^1]}=t_1-\frac{g+1}{2}t_2,\qquad\psi_1=t_1-\frac{g-1}{2}t_2.
		\end{align*}
	\end{lemma}
	\begin{proof}
		It is a straightforward computation using Lemmas~\ref{lem:classW} and~\ref{prop:classpsi1}, following the change of basis in Remark~\ref{rmk: explicit character basis Picard} and the isomorphisms~\eqref{eq: isom B2}.
	\end{proof}
	Making this change of variables, we get the following result.
	\begin{corollary}\label{cor:ChowHg1geometricbase}
		Let $g\geq2$ and $\mathrm{char}k>2g$. Then,
		\[
		\mathrm{CH}^*(\cH_{g,1})\simeq\frac{\mathbb{Z}[\psi_1,[\cW_{g,n}^1]]}{((4g+2)((g+1)\psi_1-(g-1)[\cW_{g,1}^1]),[\cW_{g,1}^1]^2+\psi_1\cdot[\cW_{g,1}^1])}
		\]			
		regardless of the parity of $g$.
	\end{corollary}
	There is another useful basis that we can consider. Notice that, depending on $g$, there are isomorphisms
	\begin{equation}\label{eq: iso GmxGm}
	\begin{tikzcd}
		(\mathbb{G}_m\times\mathbb{G}_m)/\mu_{g+1}\arrow[r,"\varphi"] & \mathbb{G}_m\times\mathbb{G}_m\arrow[r,"\varphi_{\text{even}}"]\arrow[d,"\varphi_{\text{odd}}"] & \mathbb{G}_m\times\mathbb{G}_m\\
		& \mathbb{G}_m\times\mathbb{G}_m
	\end{tikzcd}
	\end{equation}
	where
	\[
	\varphi(a,c)=(ac^{-1},c^{g+1}),\ \varphi_{\text{even}}(\lambda,\mu)=(\lambda^{g/2+1}\mu,\lambda^{g/2}\mu),\ \varphi_{\text{odd}}(\lambda,\mu)=(\lambda^{(g+1)/2}\mu,\lambda).
	\]
	The composites $\varphi_{\text{even}}\circ\varphi$ and $\varphi_{\text{odd}}\circ\varphi$ are obtained by restriction from~\eqref{eq: isom B2} and~\eqref{eq: isom GL2}.
	
	Let $l_1$ and $l_2$ be the two characters of the copy of $\mathbb{G}_m\times\mathbb{G}_m$ in the middle given by the two projections. Then, $\varphi_{\text{even}}$ and $\varphi_{\text{odd}}$ induce the two change of basis
	\begin{align*}
		t_1=\left(\frac{g}{2}+1\right)l_1+l_2,\qquad t_2=\frac{g}{2}l_1+l_2,
	\end{align*}
	when $g$ is even, and
	\begin{align*}
		t_1=\frac{g+1}{2}l_1+l_2,\qquad t_2=l_1,
	\end{align*}
	when $g$ is odd, and the equalities
	\begin{align*}
		[\cW_{g,1}^1]=l_2,\qquad\psi_1=l_1+l_2.
	\end{align*}
	This allows us to rewrite everything in another symmetric way, which again highlights the fact that the presentation does not depend on the genus, even though when seen abstractly the two rings appear different, in particular the Picard group.
	\begin{corollary}\label{cor:newbase}
		Using the change of basis above, we have
		\begin{equation*}
			\mathrm{CH}^*(\cH_{g,1})\simeq\frac{\mathbb{Z}[l_1,l_2]}{((4g+2)((g+1)l_1+2l_2),2l_2^2+l_1l_2)}.
		\end{equation*}
	\end{corollary}
	\section{The Chow ring of $\cH_{g,2}^{\far}$}\label{sec: Chow Hg2far}
	In this section, we compute the Chow ring of $\cH_{g,2}^{\far}$. To do so, we follow the strategy used for the computation of the Chow ring of $\cH_{g,1}$.
	
	Again, we set $N=2g+2$. Recall that
	\[
	\cH_{g,2}^{\far}\simeq\left[\frac{\mathbb{A}^{N+1}\setminus\Delta_{g,2}}{(\mathbb{G}_m\times\mathbb{G}_m)/\mu_{N/2}}\right]
	\]
	where the action has weights
	\[
	\begin{pmatrix}
		-N & -(N-1) & \ldots & -1 & 0 & 0 & -N/2\\
		0 & -1 & \ldots & -(N-1) & -N & -N/2 & 0
	\end{pmatrix}.
	\]
	Here, we interpret $\mathbb{A}^{N+1}$ as in Notation~\ref{not: notation n=1,2}, for $n=2$.
	
	Again, we pass to the projective setting by considering the quotient of $\mathbb{A}^{N+1}$ by the action of $\mathbb{G}_m$ with weights $(2,\ldots,2,1,1)$, hence obtaining $\cP(2^{N-1},1,1)$ and its closed subscheme $\overline{\Delta}_{g,2}$, which is the image of $\Delta_{g,2}$.
	
	Now, we construct morphisms that together will form a Chow envelope of $\overline{\Delta}_{g,2}$. What each affine space parametrizes is explained by the description on the right of the corresponding map; for instance, if we write $(g,s_1)$ this means $g(0,1)=s_1^2$, while $(g,s_2)$ means that $g(1,0)=s_2^2$.
	For every $1\leq r\leq N/2$, we define
	\[
	\begin{tikzcd}[column sep=small, row sep=tiny]
		\widetilde{a}_r:\mathbb{A}^{r+1}\times\mathbb{A}^{N+1-2r}\arrow[rr] && \mathbb{A}^{N+1}, & (f,(g,s_1,s_2))\mapsto(f^2g,f(0,1)s_1,f(1,0)s_2),\\
		\widetilde{b}_r:\mathbb{A}^r\times\mathbb{A}^{N+1-2r}\arrow[rr] && \mathbb{A}^{N+1}, & (f,(g,s_2))\mapsto(x_0^2f^2g,0,f(1,0)s_2),\\
		\widetilde{c}_r:\mathbb{A}^r\times\mathbb{A}^{N+1-2r}\arrow[rr] && \mathbb{A}^{N+1}, & (f,(g,s_1))\mapsto(x_1^2f^2g,f(0,1)s_1,0),
	\end{tikzcd}
	\]
	and for $2\leq r\leq N/2$ we define
	\[
	\begin{tikzcd}[column sep=small]
		\widetilde{d}_r:\mathbb{A}^{r-1}\times\mathbb{A}^{N+1-2r}\arrow[rr] && \mathbb{A}^{N+1}, & (f,g)\mapsto(x_0^2x_1^2f^2g,0,0).
	\end{tikzcd}
	\]
	\begin{remark}
		It is necessary to consider all these maps in order to get a Chow envelope, see~\cite[Remark 3.8]{Per22}.
	\end{remark}
	Now, we define an action of $(\mathbb{G}_m\times\mathbb{G}_m)/\mu_{g+1}$ on the source of the maps above, so that they become $(\mathbb{G}_m\times\mathbb{G}_m)/\mu_{g+1}$-equivariant.
	Then, we pass to the projective setting. First, we identify $(\mathbb{G}_m\times\mathbb{G}_m)/\mu_{g+1}$ with $T:=\mathbb{G}_m\times\mathbb{G}_m$ via the isomorphism $\varphi$ in~\eqref{eq: iso GmxGm}, so that we do not have to worry about the parity of $g$.
	
	Given elements $(a,c)\in T$, the action on the source $\mathbb{A}^{r+1}\times\mathbb{A}^{N+1-2r}$ of $\widetilde{a}_r$ is defined as
	\[
	(a,c)\cdot(f(x,y),(g(x,y),s_1,s_2))=(f(a^{-1}x,y),(c^{-2}g(a^{-1}x,y),c^{-1}s_1,a^{r-(g+1)}c^{-1}s_2)).
	\]
	The definition of the other actions is analogous. For the sake of completeness, we describe them for $\widetilde{b}_r$, $\widetilde{c}_r$ and $\widetilde{d}_r$ in this order:
	\begin{align*}
		(a,c)\cdot(f(x,y),(g(x,y),s_2))&=(f(a^{-1}x,y),(a^{-2}c^{-2}g(a^{-1}x,y),a^{r-1-(g+1)}c^{-1}s_2))\\
		(a,c)\cdot(f(x,y),(g(x,y),s_1))&=(f(a^{-1}x,y),(c^{-2}g(a^{-1}x,y),c^{-1}s_1))\\
		(a,c)\cdot(f(x,y),g(x,y))&=(f(a^{-1}x,y),a^{-2}c^{-2}g(a^{-1}x,y)).
	\end{align*}
	It is a straightforward computation to check that the maps just defined are equivariant with respect to these actions. Notice that the first component of $T$ acts trivially on the first factor of each source.
	
	Now, we pass to the projective setting; to do so, we define another action of $\Gm\times\Gm$ on each source. The first copy of $\Gm$ acts only on the first factor of each source, in the standard way. In the case of $\widetilde{a}_r$, the second copy of $\Gm$ acts on $\mathbb{A}^{N+1-2r}$ by $\mu\cdot(f,s_1,s_2)=(\mu^2f,\mu s_1,\mu s_2)$; the definition is similar for the other maps. Consider the morphism of group schemes
	\begin{equation}\label{eq: morphism groups}
	\begin{tikzcd}
		\mathbb{G}_m\times\Gm\times T\arrow[r] & T, & (\lambda,\mu,(a,c))\mapsto(a,\lambda^{-1}\mu^{-1}c).
	\end{tikzcd}
	\end{equation}
	It is again a straightforward calculation to show that the maps defined are equivariant with respect to~\eqref{eq: morphism groups}, and induce $T$-equivariant maps
	\[
	\begin{tikzcd}[row sep=tiny]
		a_r:\mathbb{P}^r\times\cP(2^{N-1-2r},1,1)\arrow[r] & \cP(2^{N-1},1,1),\\
		b_r:\mathbb{P}^{r-1}\times\cP(2^{N-2r},1)\arrow[r] & \cP(2^{N-1},1,1),\\
		c_r:\mathbb{P}^{r-1}\times\cP(2^{N-2r},1)\arrow[r] & \cP(2^{N-1},1,1),\\
		d_r:\mathbb{P}^{r-2}\times\cP(2^{N+1-2r})\arrow[r] & \cP(2^{N-1},1,1).
	\end{tikzcd}
	\]
	\begin{lemma}
		The maps $a_r$, $b_r$, $c_r$ and $d_r$ are proper and representable morphisms of quotient stacks. Moreover, they factor through $\oD_{g,2}$.
	\end{lemma}
	\begin{proof}
		The proof is the same as~\cite[Lemma 3.9]{Per22} and~\cite[Remark 3.10]{Per22}.
	\end{proof}
	\begin{proposition}\label{prop:ChowenvelopeHg2}
		Suppose $\charac k>2g$. Let $\omega_r$ be the co-product morphism of $a_r$, $b_r$, $c_r$ and $d_r$, and define $\omega=\bigsqcup_{r=1}^{g+1}\omega_r$.
		Then $\omega$ is Chow-surjective.
	\end{proposition}
	\begin{proof}
		It follows by construction that $\omega$ is a Chow envelope; the proof is analogous to~\cite[Proposition 3.12]{Per22}. Then, the statement follows from Proposition~\ref{prop: Chow envelope surjectivity with T}.
	\end{proof}
	\begin{remark}\label{rmk:Chowringsource2}
		Denote by $l_1$ and $l_2$ the two standard characters of $T$. Then,
		\begin{align*}
			\mathrm{CH}_{T}^*(\P^r\times\cP(2^{N-1-2r},1,1))&\simeq\frac{\mathbb{Z}[l_1,l_2,\xi_{r,1},\xi_{r,2}]}{(p(l_1,l_2,\xi_{r,1}),q(l_1,l_2,\xi_{r,2}))}\\
			\mathrm{CH}_{T}^*(\P^{r-1}\times\cP(2^{N-2r},1))&\simeq\frac{\mathbb{Z}[l_1,l_2,\xi_{r-1,1},\xi_{r-1,2}]}{(p'(l_1,l_2,\xi_{r-1,1}),q'(l_1,l_2,\xi_{r-1,2}))}\\
			\mathrm{CH}_{T}^*(\P^{r-2}\times\cP(2^{N+1-2r}))&\simeq\frac{\mathbb{Z}[l_1,l_2,\xi_{r-2,1},\xi_{r-2,2}]}{(p''(l_1,l_2,\xi_{r-2,1}),q''(\xi_{r-2,2}))}.
		\end{align*}
		Here, $\xi_{r,1}$ and $\xi_{r,2}$ are the pullback under the first and second projection, respectively, of the first $T$-equivariant Chern class of the tautological bundle $\cO(1)$ of the respective projective stack. In the denominator, $p$ (respectively $p'$, $p''$) is a monic polynomial in the variable $\xi_{r,1}$ (respectively $\xi_{r-1,1}$, $\xi_{r-2,1}$) of degree $r+1$ (respectively $r$, $r-1$), and similarly for $q$ (respectively $q'$, $q''$), which is of degree $N+1-2r$.
		
		Recall also that
		\[
		\mathrm{CH}_{T}^*(\cP(2^{N-1},1,1))\simeq\frac{\mathbb{Z}[l_1,l_2,h]}{(P(l_1,l_2,h))}
		\]
		for some polynomial $P(l_1,l_2,h)$ and $h=\mathrm{c}_1(\cO(1))$.
	\end{remark}
	\begin{lemma}\label{lem:hpullbackHg2}
		Using the notation above,
		\begin{align*}
			a_r^*(h)&=\xi_{r,1}+\xi_{r,2}, & b_r^*(h)&=\xi_{r-1,1}+\xi_{r-1,2},\\ c_r^*(h)&=\xi_{r-1,1}+\xi_{r-1,2}, & d_r^*(h)&=\xi_{r-2,1}+\xi_{r-2,2}.
		\end{align*}
	\end{lemma}
	\begin{proof}
		The proof is the same as the one of~\cite[Lemma 5.2]{Per22}.
	\end{proof}
	Therefore, we are left with computing the classes
	\begin{align*}
		\alpha_{r,i}&:=a_{r*}(\xi_{r,1}^i), & \beta_{r,i}&:=b_{r*}(\xi_{r-1,1}^i),\\ \gamma_{r,i}&:=c_{r*}(\xi_{r-1,1}^i), & \delta_{r,i}&:=d_{r*}(\xi_{r-2,1}^i),
	\end{align*}
	in $\CH_T^*(\cP(2^{N-1},1,1))$.
	\begin{lemma}
		The image of $\mathrm{CH}_{T}^*(\overline{\Delta}_{g,2})$ in $\mathrm{CH}_{T}^*(\cP(2^{N-1},1,1))$ is generated by
		\[
		\alpha_{r,0}\text{ for }1\leq r\leq N/2,\quad\beta_{1,0},\quad\gamma_{1,0},\quad\delta_{2,0}.
		\]
	\end{lemma}
	\begin{proof}
		Notice that there is a commutative diagram of proper morphisms
		\[
		\begin{tikzcd}
			\mathbb{P}^{r-1}\times\cP(2^{N-2r-1},1,1)\arrow[rd,"\theta"']\arrow[r,hookrightarrow,"j"] & \mathbb{P}^r\times\cP(2^{N-2r-1},1,1)\arrow[r,"a_r"] & \cP(2^{N-1},1,1)\\
			& \mathbb{P}^{r-1}\times\cP(2^{N-2r},1)\arrow[ru,"b_r"']
		\end{tikzcd}
		\]
		where $j$ is the inclusion of the hyperplane where the first coordinate of $\mathbb{P}^r$ is 0. The same reasoning as in Lemma~\ref{lem:alpha>=1} shows that $\alpha_{r,i}$ is contained in the ideal generated by $\alpha_{r,0}$ and the image of $b_{r-1*}$, for every $1\leq i\leq r$.
		
		Now, consider the commutative diagram
		\[
		\begin{tikzcd}
			\mathbb{P}^{r-1}\times\cP(2^{N-2r},1)\arrow[d,"b_r'"']\arrow[r,"b_r"] & \cP(2^{N-1},1,1)\\
			\cP(2^{N-2},1)\arrow[ru,"b_1"']
		\end{tikzcd}
		\]
		where $b_r'(f,(g,s_2))=(f^2g,f(1,0)s_2)$; recall that here $(g,s_2)\in\cP(2^{N-2r},1)$ satisfies $g(1,0)=s_2^2$. This shows that $b_r$ factors through $b_1$. Moreover, the $T$-equivariant pullback $b_1^*$ is surjective; by the projection formula it follows that the image of $b_{1*}$ is generated by $b_{1*}(1)=\beta_{1,0}$.
		
		An analogous argument works for $\gamma_{r,i}$ and $\delta_{r,i}$. The statement follows.
	\end{proof}
	Let $\phi:\cP(2^{N-1},1,1)\rightarrow\cP(2^N,1)$ be the morphism induced by $s_2\mapsto s_2^2$; equivalently, $\phi(f,s_1,s_2)=(f,s_1)$. Notice that it is representable and flat. The proof of the following Lemma is straightforward.
	\begin{lemma}
		Interpret $\cP(2^N,1)$ as parametrizing pairs $(f,s_1)$ with $f(0,1)=s_1^2$, and let $\mathbb{A}^N\subset\cP(2^N,1)$ be the open subscheme where $s_1$ does not vanish. Then,
		\[
		\begin{tikzcd}
			\mathbb{A}^N\arrow[r]\arrow[d,"u"] & \cP(2^{N-1},1,1)\arrow[d,"\phi"]\\
			\mathbb{A}^N\arrow[r] & \cP(2^N,1)
		\end{tikzcd}
		\]
		is cartesian, where $u(m_1,\ldots,m_{N-1},s_2)=(s_2^2,m_1,\ldots,m_{N-1})$.
	\end{lemma}	
	\begin{notation}\label{notation: hyperplanes n=2}
		We adopt the same notation as Notation~\ref{notation: hyperplanes n=1}. Also, we define coordinates $m_1,\ldots,m_{N-1},s_1,s_2$ on $\cP(2^{N-1},1,1)$ so that every element $(f,s_1,s_2)$ has $f$ of the usual form $f=\sum_im_ix^{N-i}y^i$, with $m_N=s_1^2$ and $m_0=s_2^2$.
	\end{notation}
	\begin{lemma}
		Let $i:\oD_{g,1}\hookrightarrow\cP(2^N,1)$ be the inclusion of the locus of singular binary forms. For all $r$, the classes $\alpha_{r,0}$ are in the image of $\phi^*i_*$.
	\end{lemma}
	\begin{proof}
		Construct the commutative diagram with cartesian square
		\[
		\begin{tikzcd}
			\mathbb{P}^r\times\cP(2^{N-2r-1},1,1)\arrow[rd]\arrow[rr,bend left=20,"a_r"]\arrow[r,"a_r'"] & X\arrow[r]\arrow[d,"\psi"] & \cP(2^{N-1},1,1)\arrow[d,"\phi"]\\
			& \mathbb{P}^r\times\cP(2^{N-2r},1)\arrow[r] & \cP(2^N,1).
		\end{tikzcd}
		\]
		Now, the set-theoretic image of $a_r$ is the preimage under $\phi$ of the image of the bottom morphism. We prove that $a'_{r*}(1)=\psi^*(1)=1$. It is enough to show it after base changing the above diagram to an open substack of $\cP(2^N,1)$ intersecting the image of $a_r$. Choose the complement of $(s_1=0)\cup(m_0=0)$ in $\cP(2^N,1)$, where $(m_0=0)$ is then the hyperplane of points $(f,s_1)$ such that $f(1,0)=0$. Over that substack (actually, scheme) we have that $a_r'$ is an isomorphism, thanks to the previous Lemma. Indeed, over $s_1\not=0$ we have a diagram with cartesian square
		\[
		\begin{tikzcd}
			\mathbb{A}^r\times\mathbb{A}^{N-2r}\arrow[r,"a_r'"] & X^{\text{o}}\arrow[r]\arrow[d] & \mathbb{A}^N\arrow[d,"u"]\\
			& \mathbb{A}^r\times\mathbb{A}^{N-2r}\arrow[r] & \mathbb{A}^N
		\end{tikzcd}
		\]
		inducing an isomorphism
		\[
		X^{\text{o}}\simeq\Spec k[x_0,\ldots,x_{r-1},y_0,\ldots,y_{N-2r-1},t]/(t^2-x_0^2y_0).
		\]
		Then, where $x_0\not=0$, which is over the complement of $(m_0=0)$, the morphism $a_r'$ is an isomorphism, and the statement follows.
	\end{proof}
	Let $\phi':\cP(2^{N-1},1,1)\rightarrow\cP(2^N,1)$ the arrow induced by $s_1\mapsto s_1^2$.
	\begin{lemma}
		The classes $\beta_{1,0}$, $\gamma_{1,0}$ and $\delta_{2,0}$ satisfy
		\begin{align*}
			\beta_{1,0}=[(s_1=0)]\cdot[(m_{N-1}=0)], && \gamma_{1,0}=[(s_2=0)]\cdot[(m_1=0)], && \delta_{2,0}=\beta_{1,0}\cdot\gamma_{1,0}.
		\end{align*}
		In particular, they are obtained from the relations in $\cH_{g,1}$ by pullback along $\phi$, $\phi'$.
	\end{lemma}
	\begin{proof}
		It follows immediately from the definition of the maps.
	\end{proof}
	\begin{remark}\label{rmk:projectivebundlerelationHg2far}
		As in Remark~\ref{rmk:projectivebundlerelationHg1}, the above Lemma also shows that the relation in $\mathrm{CH}^*_T(\cP(2^{N-1},1,1))$ coming from the projective bundle formula is contained in the ideal generated by $\beta_{1,0}$.
	\end{remark}	
	Putting everything together we get the following Proposition.
	\begin{proposition}\label{prop: Chow Hg2far}
		Let $\charac k>2g$. Then,
		\begin{align*}
			\mathrm{CH}^*(\cH_{g,2}^{\far})\simeq\frac{\mathbb{Z}[l_1,l_2]}{((4g+2)((g+1)l_1+2l_2),2l_2^2+l_1l_2,(g+1)(2g+1)l_1^2+(4g+2)l_1l_2)}.
		\end{align*}
	\end{proposition}
	It is useful to write it also in terms of $[\cW_{g,2}^1]$ and $\psi_1$.
	\begin{corollary}\label{cor: Chow Hg2far}
		Let $\charac k>2g$. Then, $\CH^*(\cH_{g,2}^{\far})$ is generated by $\psi_1$ and $[\cW_{g,2}^1]$, and the ideal of relations is generated by
		\addtolength{\jot}{2pt}
		\begin{align*}
			\bullet & \
			(4g+2)((g+1)\psi_1-(g-1)[\cW_{g,2}^1]),\\
			\bullet & \
			[\cW_{g,2}^1]^2+\psi_1\cdot[\cW_{g,2}^1],\qquad(2g+1)(g+1)\psi_1^2+(2g+1)(3g-1)[\cW_{g,2}^1]^2.
		\end{align*}
	\end{corollary}
	\begin{proof}
		It follows by performing the change of basis $[\cW_{g,2}^1]=l_2$ and $\psi_1=l_1+l_2$, see the proof of Corollary~\ref{cor:newbase}.
	\end{proof}
	\begin{remark}
		The new relation in degree 2 can be described as follows. Notice that
		\begin{align*}
			[\cW_{g,2}^2]=(g+1)l_1+l_2, && \psi_2=gl_1+l_2,
		\end{align*}
		hence,
		\begin{align*}
			[\cW_{g,2}^2]\cdot([\cW_{g,2}^2]+\psi_2)&=((g+1)l_1+l_2)((2g+1)l_1+2l_2)\\
			&=(g+1)(2g+1)l_1^2+2l_2^2+(4g+3)l_1l_2,
		\end{align*}
		which is the new relation.
	\end{remark}
	\section{Generators and some relations of the Chow ring of $\cH_{g,n}$}\label{sec: generators Chow}
	We have very explicit generators of the Picard group of $\cH_{g,n}$, with evident geometric meaning. Together with our concrete description of $\cH_{g,n}$, this should decisively help in the computation of the integral Chow ring $\cH_{g,n}$, at least when this is generated in degree 1. We know that this does not always happen, since the (rational) Chow ring of the stack of $n$-pointed hyperelliptic curves is not always tautological. An example is given in~\cite[Theorem 3]{PG01}, where the authors construct a non-tautological class in $\cM_{2,20}$. However, we shall prove that for $1\leq n\leq2g+3$ the Chow ring is generated in degree 1. Therefore, Proposition~\ref{prop: generators Picard} gives us also the generators of the Chow ring, and we are left with computing the intersection of those divisors. Thanks to their geometric nature, this is doable.
	\subsection{Generators}
	We show that the generators of the Picard group described in Propositions~\ref{prop: generators Picard} and~\ref{prop: generators Picard n=1,2} actually generate the Chow ring for $1\leq n\leq2g+3$. We already know this for $n=1$, see Section~\ref{sec:ChowringHg1}, and we want to use induction on $n$. To do so, we leverage the fact that $\cZ_{g,n}^{i,j}$ is isomorphic to an open substack of $\cH_{g,n-1}$.
	\begin{definition}\label{def:Hgnij}
		Let $1\leq i<n$. We define $\cH_{g,n-1}^{i}$ as the open substack of $\cH_{g,n-1}$
		\[
			\cH_{g,n-1}^{i}:=\cH_{g,n-1}\setminus(\cW_{g,n-1}^i\cup\bigcup_{l\not=i}\cZ_{g,n-1}^{i,l}).
		\]
	\end{definition}
	\begin{remark}\label{rmk: uij}
		Notice that for every $1\leq i<j\leq n$ there is a morphism
		\[
		\begin{tikzcd}
			u_{i,j}:\cH_{g,n-1}^i\arrow[r] & \cH_{g,n}
		\end{tikzcd}
		\]
		defined by $(C\rightarrow S,\widetilde{\sigma}_1,\ldots,\widetilde{\sigma}_{n-1})\mapsto(C\rightarrow S,\widetilde{\sigma}_1,\ldots,\widetilde{\sigma}_{j-1},\alpha\circ\widetilde{\sigma}_i,\widetilde{\sigma}_j,\ldots,\widetilde{\sigma}_{n-1})$.
		This is well defined as the images of $\alpha\circ\widetilde{\sigma}_i$ and $\widetilde{\sigma}_l$ are disjoint for every $l$.
	\end{remark}
	\begin{lemma}\label{lem: Zgnij as open substack}
		Let $1\leq i<j\leq n$. The morphism
		\[
		\begin{tikzcd}
			\Phi_{i,j}:\cZ_{g,n}^{i,j}\arrow[r] & \cH_{g,n-1}
		\end{tikzcd}
		\]
		that forgets the $j$-th section is an open immersion, and induces $\cZ_{g,n}^{i,j}\simeq\cH_{g,n-1}^{i}$, with inverse $u_{i,j}$.
	\end{lemma}
	\begin{proof}
		For every $i\not=l\not=j$, we have $\cW_{g,n}^i\cap\cZ_{g,n}^{i,j}=\cZ_{g,n}^{i,l}\cap\cZ_{g,n}^{i,j}=\emptyset$, thus $\Phi_{i,j}$ factors through $\cH_{g,n-1}^{i}\subset\cH_{g,n-1}$. Then, it is clear that $u_{i,j}$ and $\Phi_{i,j}$ are inverses.
	\end{proof}
	\begin{lemma}
		Let $1\leq n\leq2g+3$. Then the Chow ring of $\cH_{g,n}^{\far}$ is generated in degree 1. In particular, for $n\geq3$ it is generated by $[\cW_{g,n}^1]$.
	\end{lemma}
	\begin{proof}
		We know that this is true for $n=1,2$, see Theorem~\ref{thm:ChowHg1} and Proposition~\ref{prop: Chow Hg2far}. For $3\leq n\leq 2g+3$, thanks to Proposition~\ref{prop:representationweightedprojectivestack3<=n<=2g+3}, we know that $\cH_{g,n}^{\far}$ is an open substack of a weighted projective space, and the localization sequence concludes. The last part follows from the first and Subsection~\ref{subsec:Picardgroup}.
	\end{proof}
	We are now ready to prove the main Proposition of this Subsection.
	\begin{proposition}\label{prop:generateddegree1}
		For all $1\leq n\leq2g+3$, $\mathrm{CH}^*(\cH_{g,n})$ is generated in degree 1.
	\end{proposition}
	\begin{proof}			
		We already know that the statement holds for $n=1$ and for $\cH_{g,n}^{\far}$. We prove the Proposition by induction on $n$.
		Suppose $n\geq2$ and that the statement is true for $n-1$. Let $R_n$ be the subring of $\CH^*(\cH_{g,n})$ generated by 1 and the classes of degree 1. By Lemma~\ref{lem: Zgnij as open substack} there is an exact sequence
		\begin{equation}\label{diag:generateddeg1}
			\begin{tikzcd}
				\bigoplus_{i<j}\mathrm{CH}^*(\cH_{g,n-1}^{i})\arrow[rr,"\sum(u_{i,j})_*"] && \mathrm{CH}^*(\cH_{g,n})\arrow[r,"\iota^*"]& \mathrm{CH}^*(\cH_{g,n}^{\far})\arrow[r] & 0
			\end{tikzcd}
		\end{equation}
		where $u_{i,j}$ is as in Remark~\ref{rmk: uij}. Since $\mathrm{CH}^*(\cH_{g,n-1})$ is generated in degree 1 by the classes in Proposition~\ref{prop: generators Picard}, the same holds for the Chow ring of $\cH_{g,n-1}^{i}$. Therefore, every element of $\CH^*(\cH_{g,n}^i)$ is the pullback along $u_{i,j}$ of an element of $R_n\subset\CH^*(\cH_{g,n})$. It follows from the projection formula that the image of $\sum(u_{i,j})_*$ is contained in $R_n$.
		
		Now, let $\beta$ be an element of $\mathrm{CH}^*(\cH_{g,n})$.
		Since $[\cW_{g,n}^1]$ and $\psi_1$ generate $\mathrm{CH}^*(\cH_{g,n}^{\far})$, we can modify $\beta$ with an element generated in degree 1 so that its restriction to $\cH_{g,n}^{\far}$ is 0. Therefore, from now on we may assume that $\iota^*(\beta)=0$. Thanks to the exactness of the sequence~\eqref{diag:generateddeg1}, $\alpha$ is in the image of $\sum_{i,j}(u_{i,j})_*$, thus in $R_n$.
	\end{proof}
	\begin{corollary}\label{cor:generators}
		Assume that $\charac k>2g$. Then, the classes $\psi_1$, $[\cW_{g,2}^1]$, $[\cZ_{g,2}^{1,2}]$ form a minimal set of generators of $\mathrm{CH}^*(\cH_{g,2})$, and
		\begin{equation}
			(4g+2)(g+1)\psi_1=(4g+2)(g-1)[\cW_{g,2}^1]
		\end{equation}
		is the only relation in degree 1.
		
		Let $n\geq3$. Then
		\[
		\{[\cZ_{g,n}^{1,j}]\}_{2\leq j\leq n}\cup\{[\cZ_{g,n}^{2,3}]\}\cup\{[\cW_{g,n}^1]\}
		\]
		is a minimal set of generators of $\mathrm{CH}^*(\cH_{g,n})$, and
		\begin{equation}
			(8g+4)[\cW_{g,n}^1]=(4g+2)(g+1)([\cZ_{g,n}^{1,2}]+[\cZ_{g,n}^{1,3}]-[\cZ_{g,n}^{2,3}])
		\end{equation}
		is the only relation in degree 1.
	\end{corollary}
	\begin{remark}
		Notice that the proof of Proposition~\ref{prop:generateddegree1} works for all $n$ such that $\mathrm{CH}^*(\cH_{g,n}^{\far})$ and $\mathrm{CH}^*(\cH_{g,n-1})$ are generated in degree 1. Thus, it shows that the first examples of non-tautological classes are to be looked for in $\mathrm{CH}^*(\cH_{g,n}^{\far})$.
	\end{remark}
	\subsection{Intersection of Chow classes}
	In this Subsection we will compute some relations between intersections of the divisors appearing in Corollary~\ref{cor:generators} and other related divisors. These will completely determine the Chow ring of $\cH_{g,n}$ for $n=2$, and only partially for $3\leq n\leq2g+3$. To better structure the Subsection, we subdivide the relations we get according to their degree, which is 1 or 2.
	\subsubsection{Degree 1}\label{subsubsec: degree 1}
	We start with degree 1 relations. On the one hand, we already know all of them between the generators, see Corollary~\ref{cor:generators}; on the other hand, there are other geometrically meaningful classes that we want to write in terms of the generators we have. This will be important to find new relations in higher degrees.
	\begin{remark}\label{rmk: rationally + far = integrally}
		We say that a relation between Chow classes holds rationally if it holds in $\CH_{\Q}^*(\cH_{g,n})$, and same for the Picard group. Since the restriction $\Pic(\cH_{g,n})_{\text{tors}}\rightarrow\Pic(\cH_{g,n}^{\far})_{\text{tors}}$ is injective by~\cite[Theorem 3.2 and Propositions 3.3, 3.5]{Lan23}, to show that a relation holds in $\Pic(\cH_{g,n})$ it is enough to show that it holds rationally and in $\Pic(\cH_{g,n}^{\far})$.
	\end{remark}
\begin{lemma}\label{lem:gW1+W2}
	Let $n\geq2$. Then,
	\begin{equation}\label{eq:gW1+W2}
	g[\cW_{g,n}^1]+[\cW_{g,n}^2]=(g+1)\psi_1+(g+1)[\cZ_{g,n}^{1,2}].
	\end{equation}
\end{lemma}
\begin{proof}
	By Remark~\ref{rmk: rationally + far = integrally}, it is enough to show that~\eqref{eq:gW1+W2} holds rationally and in $\Pic(\cH_{g,n}^{\far})$. Thanks to the rational relations $[\cW_{g,n}^1]+[\cW_{g,n}^2]=(g+1)[\cZ_{g,n}^{1,2}]$, proved in~\cite[Theorems 1.1, 1.2]{EH21} or~\cite[Lemma 5.10]{Lan23}, and  $(g-1)[\cW_{g,n}]=(g+1)\psi_1$ coming from Proposition~\ref{prop: generators Picard n=1,2}, the relation holds in $\Pic_{\Q}(\cH_{g,n})$. From the computation in Subsection~\ref{subsec:Picardgroup} of the classes $\psi_1$, $[\cW_{g,n}^1]$, $[\cW_{g,n}^2]$ in $\Pic(\cH_{g,n}^{\far})$, the equation also holds in $\Pic(\cH_{g,n}^{\far})$.
\end{proof}
\begin{lemma}\label{lem:psi2}
	Using the basis described in Remark~\ref{rmk: explicit character basis Picard}, in $\Pic(\cH_{g,2}^{\far})$ we have
	\[
	\psi_2=
	\begin{cases}
		((g-1)/2,1)\in\mathbb{Z}\oplus\mathbb{Z}/(8g+4)\mathbb{Z} & \text{ if }g\text{ is odd},\\
		(g-1,-g/2+1)\in\mathbb{Z}\oplus\mathbb{Z}/(4g+2)\mathbb{Z} & \text{ if }g\text{ is even.}
	\end{cases}
	\]
	Moreover, for every $n\geq2$, in $\Pic(\cH_{g,n})$ it holds
	\begin{align}
		(g-1)([\cW_{g,n}^1]-[\cW_{g,n}^2])&=(g+1)(\psi_1-\psi_2),\\
		[\cW_{g,n}^1]+[\cW_{g,n}^2]&=\psi_1+\psi_2+2[\cZ_{g,n}^{1,2}],\\
		(g-1)[\cW_{g,n}^1]+\psi_2&=g\psi_1+(g-1)[\cZ_{g,n}^{1,2}].
	\end{align}
\end{lemma}
\begin{proof}
	The first part follows from the computation in Proposition~\ref{prop:classpsi1}, by pulling back along the map $\cH_{g,2}^{\far}\rightarrow\cH_{g,1}$ that forgets the first section. For the second part, we can assume $n=2$. Then, the equations follow from the fact that the relations hold both rationally and in $\CH(\cH_{g,2}^{\far})$, see Remark~\ref{rmk: rationally + far = integrally}.
\end{proof}
Now, we focus on the case $n\geq3$.
\begin{lemma}\label{lem:W1W2integral}
	Let $n\geq3$. Then,
	\begin{equation}\label{eq:W1W2integral}
		[\cW_{g,n}^1]-[\cW_{g,n}^2]=(g+1)([\cZ_{g,n}^{1,3}]-[\cZ_{g,n}^{2,3}]).
	\end{equation}
\end{lemma}
\begin{proof}
	The fact that it holds rationally follows from Proposition~\ref{prop: generators Picard} and the rational relation
	\[
	[\cW_{g,n}^1]+[\cW_{g,n}^2]=(g+1)[\cZ_{g,n}^{1,2}]
	\]
	proved in~\cite[Theorems 1.1, 1.2]{EH21}, or in~\cite[Lemma 5.10]{Lan23}. By Remark~\ref{rmk: rationally + far = integrally}, it is enough to show the vanishing of $[\cW_{g,n}^1]-[\cW_{g,n}^2]$ in $\Pic(\cH_{g,n}^{\far})$, and we can assume $n=3$. Then, the result follows from Lemma~\ref{lem: pullback Hg3far Hg2far}.
\end{proof}
\begin{lemma}\label{lem: W1 psi1 Z}
	Let $n\geq3$. Then,
	\begin{equation}\label{eq:W1psi1}
		[\cW_{g,n}^1]=\psi_1+[\cZ_{g,n}^{1,2}]+[\cZ_{g,n}^{1,3}]-[\cZ_{g,n}^{2,3}].
	\end{equation}
\end{lemma}
\begin{proof}
	By Lemma~\ref{lem: pullback Hg3far Hg2far}, we know that $\psi_1-[\cW_{g,n}^1]$ is 0 in $\Pic(\cH_{g,n}^{\far})$ for $n\geq3$. By Lemmas~\ref{lem:gW1+W2} and~\ref{lem:W1W2integral}, we have
	\[
	(g+1)[\cW_{g,n}^1]=(g+1)\psi_1+(g+1)([\cZ_{g,n}^{1,2}]+[\cZ_{g,n}^{1,3}]-[\cZ_{g,n}^{2,3}]),
	\]
	hence equation~\eqref{eq:W1psi1} holds rationally. We conclude by Remark~\ref{rmk: rationally + far = integrally}.
\end{proof}
	\subsubsection{Degree 2}\label{subsubsec: degree 2}
	We use the geometry of $\cH_{g,n}$ and the computations above to obtain various identities in $\CH^2(\cH_{g,n})$. We start with a geometrically intuitive vanishing.
	\begin{lemma}\label{lem:ortogonality}
		Let $n\geq2$. In $\mathrm{CH}^*(\cH_{g,n})$, we have
		\begin{align}
			[\cZ_{g,n}^{1,i}]\cdot[\cW_{g,n}^1]&=0\quad\text{for all }i\not=1.
		\end{align}
		If $n\geq3$, then
		\begin{align}
			[\cZ_{g,n}^{1,i}]\cdot[\cZ_{g,n}^{1,j}]&=0\quad\text{for all }1<i<j\leq n,\\
			[\cZ_{g,n}^{1,i}]\cdot[\cZ_{g,n}^{2,3}]&=0\quad\text{for } i=2,3.
		\end{align}
	\end{lemma}
	\begin{proof}
		It follows from the fact that the corresponding intersections are empty.
	\end{proof}
	The next three Lemmas are particularly important for the case $n=2$.
	\begin{lemma}\label{lem: W1 psi1}
		Let $n\geq i\geq1$. Then,
		\[
			[\cW_{g,n}^i]^2+\psi_i\cdot[\cW_{g,n}^i]=0.
		\]
	\end{lemma}
	\begin{proof}
		For $i=n=1$, this is proved in Corollary~\ref{cor:ChowHg1geometricbase}, and the other cases are obtained by pullback along forgetful morphisms.
	\end{proof}
	\begin{lemma}\label{lem:psiZ23}
		Let $n\geq2$. Then, for all $i\not=j$,
		\[
		[\cZ_{g,n}^{i,j}]|_{\cZ_{g,n}^{i,j}}=-\psi_i|_{\cZ_{g,n}^{i,j}}=-\psi_j|_{\cZ_{g,n}^{i,j}}
		\]
		in $\Pic(\cZ_{g,n}^{i,j})$.
		In particular, in $\CH^*(\cH_{g,n})$ we have
		\[
		[\cZ_{g,n}^{i,j}]^2=-[\cZ_{g,n}^{i,j}]\cdot\psi_i=-[\cZ_{g,n}^{i,j}]\cdot\psi_j.
		\]
	\end{lemma}
	\begin{proof}
		Clearly, it is enough to show it for $n=2$. Let
		\[
		\begin{tikzcd}
			\cC_{g,2}\arrow[r,"f"] & \cP_{g,2}\arrow[r,"\pi"] & \cH_{g,2}
		\end{tikzcd}
		\]
		be the universal family, with universal sections $\widetilde{\sigma}_1$, $\widetilde{\sigma}_2$, and universal `intermediate' sections $\sigma_1:=f\circ\widetilde{\sigma}_1$ and $\sigma_2:=f\circ\widetilde{\sigma}_2$. Let $L_i^\vee$ be the ideal sheaf associated to $\sigma_i(\cH_{g,2})=:\cD_{g,2}^i\subset\cP_{g,2}$.  Thanks to~\cite[Lemma 2.3]{Lan23}, we have $\sigma_2^*\cO_{\cD_{g,2}^1}=\cO_{\cZ_{g,2}^{1,2}}$. Applying $\sigma_2^*$ to the exact sequence
		\[
		\begin{tikzcd}
			0\arrow[r] & L_1^\vee\arrow[r] & \cO_{\cP_{g,2}}\arrow[r] & \cO_{\cD_{g,2}^1}\arrow[r] & 0
		\end{tikzcd}
		\]
		we get a surjective map $\sigma_2^*(L_1^\vee)\rightarrow\cO_{\cH_{g,2}}(-\cZ_{g,2}^{1,2})$ between invertible sheaves, hence an isomorphism. It follows that \[
		\cO_{\cH_{g,2}}(\cZ_{g,2}^{1,2})\simeq\sigma_2^*L_1\simeq\sigma_1^*L_2.
		\]
		
		Let $\cC:=\cC_{g,2}$, $\cP:=\cP_{g,2}$ and $\cH:=\cH_{g,2}$. Recall that there is an exact sequence of sheaves
		\[
		\begin{tikzcd}
			0\arrow[r] & f^*\omega_{\cP/\cH}\arrow[r] & \omega_{\cC/\cH}\arrow[r] & \Omega_{\cC/\cP}\arrow[r] & 0
		\end{tikzcd}
		\]
		where $\Omega_{\cC/\cP}$ is supported at the Weierstrass divisor. Let $u_{1,2}:\cZ_{g,2}^{1,2}\hookrightarrow\cH_{g,2}$ be the inclusion. Applying $u_{1,2}^*\circ\widetilde{\sigma}_1^*$ to the above exact sequence, we get
		\[
		\begin{tikzcd}
			u_{1,2}^*\sigma_1^*(\omega_{\cP/\cH})\arrow[r] & u_{1,2}^*\widetilde{\sigma}_1^*(\omega_{\cC/\cH})\arrow[r] & u_{1,2}^*\widetilde{\sigma}_1^*(\Omega_{\cC/\cP})\arrow[r] & 0.
		\end{tikzcd}
		\]
		Notice that the class of the invertible sheaf in the middle is equal to the restriction of $\psi_1$ to $\cZ_{g,2}^{1,2}$. Since the image of $\widetilde{\sigma}_1\circ u_{1,2}$ does not intersect the Weierstrass divisor, $u_{1,2}^*\widetilde{\sigma}_1^*\Omega_{\cC/\cP}=0$, and we get an isomorphism
		\[
		\begin{tikzcd}
			u_{1,2}^*\sigma_1^*\omega_{\cP/\cH}\arrow[r,"\simeq"] & u_{1,2}^*\widetilde{\sigma}_1^*\omega_{\cC/\cH}.
		\end{tikzcd}
		\]
		Now, $\sigma_1^*\omega_{\cP/\cH}\simeq\sigma_1^*(L_1^\vee)$, and clearly $u_{1,2}^*\sigma_1^*\simeq u_{1,2}^*\sigma_2^*$. It follows that
		\[
		u_{1,2}^*\cO_{\cH_{g,2}}(-\cZ_{g,2}^{1,2})\simeq u_{1,2}^*\sigma_2^*L_1^\vee\simeq u_{1,2}^*\sigma_1^*L_1^\vee\simeq u_{1,2}^*\sigma_1^*\omega_{\cP/\cH}\simeq u_{1,2}^*\widetilde{\sigma}_1^*(\omega_{\cC/\cH})
		\]
		which proves the Lemma.
	\end{proof}
		\begin{lemma}\label{lem:longrel23}
		For all $n\geq2$, we have
		\begin{equation}\label{eq:longrel23}
		(2g+1)(g+1)\psi_1^2+(2g+1)(3g-1)[\cW_{g,n}^1]^2=(2g+1)(g+1)[\cZ_{g,n}^{1,2}]^2.
		\end{equation}
	\end{lemma}
	\begin{proof}
		By Lemma~\ref{lem: W1 psi1} we have $[\cW_{g,2}^2]^2+\psi_2[\cW_{g,2}^2]=0$. Writing this relation in terms of $\psi_1$, $[\cZ_{g,2}^{1,2}]$ and $[\cW_{g,2}^1]$ using Lemma~\ref{lem:gW1+W2} and the last relation of Lemma~\ref{lem:psi2}, we get equation~\eqref{eq:longrel23}.
	\end{proof}
	Now, we focus on the case $n\geq3$.
	\begin{lemma}\label{lem:W1Z23}
		Let $n\geq3$. Then,
		\begin{align}
			[\cW_{g,n}^1]\cdot[\cZ_{g,n}^{2,3}]=-(g+1)[\cZ_{g,n}^{2,3}]^2,\\
			(2g+1)(g+1)[\cZ_{g,n}^{1,2}]^2=(2g+1)(g+1)[\cZ_{g,n}^{i,j}]^2, && \text{ for every }i\not=j.
		\end{align}
	\end{lemma}
	\begin{proof}
		To get the first relation, multiply equation~\eqref{eq:W1W2integral} of Lemma~\ref{lem:W1W2integral} by $[\cZ_{g,n}^{2,3}]$, and apply Lemma~\ref{lem:ortogonality}. For the second, take equation~\eqref{eq:longrel23} and use Lemma~\ref{lem: W1 psi1 Z} together with the first part of this Lemma.
	\end{proof}
	\begin{lemma}\label{lem:Z23square}
		Let $n\geq3$. Then, for all $i\not=j$,
		\begin{align}
			(4g+2)(g+1)[\cZ_{g,n}^{i,j}]^2=0,\\
			(4g+2)[\cZ_{g,n}^{2,3}]\cdot[\cW_{g,n}^1]=0,\\
			(8g+4)[\cW_{g,n}^1]^2=0.
		\end{align}
	\end{lemma}
	\begin{proof}
		For the first equation, it is enough to consider the case $(i,j)=(1,2)$, which follows from Lemma~\ref{lem:ortogonality}, multiplying
		\[
		(8g+4)[\cW_{g,n}^1]=(4g+2)(g+1)([\cZ_{g,n}^{1,2}]+[\cZ_{g,n}^{1,3}]-[\cZ_{g,n}^{2,3}])
		\]
		by $[\cZ_{g,n}^{1,2}]$. The second relation follows from the first and Lemma~\ref{lem:W1Z23}. The third equation follows from the second, multiplying the relation above by $[\cW_{g,n}^{1}]$.
	\end{proof}		
	\begin{lemma}\label{lem:W1square}
		Let $n\geq3$. Then,
		\begin{equation}\label{eq:W1square}
			2[\cW_{g,n}^1]^2=-[\cW_{g,n}^1]\cdot[\cZ_{g,n}^{2,3}]=(g+1)[\cZ_{g,n}^{2,3}]^2.
		\end{equation}
	\end{lemma}
	\begin{proof}
		From Lemma~\ref{lem: W1 psi1 Z} and Lemma~\ref{lem:W1Z23}, we know that
		\[
		[\cW_{g,n}^1]^2=[\cW_{g,n}^1]\cdot\psi_1-[\cW_{g,n}^1]\cdot[\cZ_{g,n}^{2,3}]=[\cW_{g,n}^1]\cdot\psi_1+(g+1)[\cZ_{g,n}^{2,3}]^2.
		\]
		As $[\cW_{g,n}^1]\cdot\psi_1=-[\cW_{g,n}]^2$ by Lemma~\ref{lem: W1 psi1}, this concludes.
	\end{proof}
	For $n>3$ there are more relations.
	\begin{lemma}\label{lem:Z1isquarei>3}
		Let $n\geq i>3$. Then,
		\begin{equation}
			[\cZ_{g,n}^{1,i}]^2=-[\cZ_{g,n}^{1,i}]\cdot[\cZ_{g,n}^{2,3}]=[\cZ_{g,n}^{2,3}]^2.
		\end{equation}
	\end{lemma}
	\begin{proof}
		By Proposition~\ref{prop:relationsZ},
		\[
		[\cZ_{g,n}^{1,i}]=-[\cZ_{g,n}^{2,3}]+[\cZ_{g,n}^{1,2}]+[\cZ_{g,n}^{3,i}].
		\]
		Multiplying by $[\cZ_{g,n}^{1,i}]$ and $[\cZ_{g,n}^{2,3}]$ we get the two equalities, respectively.
	\end{proof}
	\begin{lemma}\label{lem:Z1isquareall}
		Let $n\geq5$ and $i>3$, then
		\begin{equation}
			[\cZ_{g,n}^{1,2}]^2=[\cZ_{g,n}^{1,3}]^2=[\cZ_{g,n}^{1,i}]^2=-[\cZ_{g,n}^{1,i}]\cdot[\cZ_{g,n}^{2,3}]=[\cZ_{g,n}^{2,3}]^2.
		\end{equation}
	\end{lemma}
	\begin{proof}
		It follows from Lemma~\ref{lem:Z1isquarei>3} and the fact that the relation
		\[
		[\cZ_{g,n}^{1,4}]^2=[\cZ_{g,n}^{1,5}]^2
		\]
		holds even after permuting the indices.
	\end{proof}
	\section{Computation of $\mathrm{CH}^*(\cH_{g,2})$}\label{sec: Chow Hg2}
	In this Section we completely compute the Chow ring of $\cH_{g,2}$. In order to do this, we consider the exact sequence
	\[
	\begin{tikzcd}
		\mathrm{CH}^*(\cH_{g,1}^{1})\arrow[r,"(u_{1,2})_*"] & \mathrm{CH}^*(\cH_{g,2})\arrow[r] & \mathrm{CH}^*(\cH_{g,2}^{\far})\arrow[r] & 0
	\end{tikzcd}
	\]
	where $\cH_{g,1}^{1}=\cH_{g,1}\setminus\cW_{g,1}^1$, see Definition~\ref{def:Hgnij}. In Section~\ref{sec: Chow Hg2far}, we have obtained the Chow ring of $\cH_{g,2}^{\far}$, see Corollary~\ref{cor: Chow Hg2far}. The ring on the left is readily computed.
	\begin{lemma}\label{lem: Chow Z12 n=2}
		Suppose $\charac k>2g$. Then,
		\[
		\mathrm{CH}^*(\cZ^{1,2}_{g,2})\simeq\mathrm{CH}^*(\cH_{g,1}^{1})\simeq\frac{\mathbb{Z}[\psi_1]}{((4g+2)(g+1)\psi_1)}.
		\]
	\end{lemma}
	\begin{proof}
		It follows immediately from the localization sequence and Corollary~\ref{cor:ChowHg1geometricbase}.
	\end{proof}
	\begin{theorem}\label{thm:ChowHg2}
		Let $\charac k>2g$. Then, $\mathrm{CH}^*(\cH_{g,2})$ is generated by the classes of $\cW_{g,2}^1$, $\cZ_{g,2}^{1,2}$ and $\psi_1$, and the ideal of relations is generated by
		\addtolength{\jot}{2pt}
		\begin{align*}
			\bullet & \ (4g+2)((g+1)\psi_1-(g-1)[\cW_{g,2}^1]),\\
			\bullet & \  [\cZ_{g,2}^{1,2}]^2+\psi_1\cdot[\cZ_{g,2}^{1,2}],\qquad[\cW_{g,2}^{1}]^2+\psi_1\cdot[\cW_{g,2}^1],\qquad
			[\cZ_{g,2}^{1,2}]\cdot[\cW_{g,2}^1],\\
			\bullet & \  (2g+1)(g+1)\psi_1^2+(2g+1)(3g-1)[\cW_{g,2}^1]^2-(2g+1)(g+1)[\cZ_{g,2}^{1,2}]^2.
		\end{align*}
	\end{theorem}
	\begin{proof}
		We already know that those are the generators, and that every relation in degree 1 is a multiple of the first row in the statement. Moreover, by Section~\ref{sec: generators Chow}, we already know that all the relations above hold.
		
		Now, suppose we have a new relation in degree $q\geq2$. Using the middle row, we get a vanishing linear combination of $\psi_1^q$, $[\cW_{g,2}^1]^q$ and $[\cZ_{g,2}^{1,2}]^q$. Modulo $[\cZ_{g,2}^{1,2}]$, we get a relation in $\mathrm{CH}^*(\cH_{g,2}^{\far})$. By Corollary~\ref{cor: Chow Hg2far}, all the relations in $\CH^*(\cH_{g,2}^{\far})$ are restriction of those appearing in this Theorem; it follows that it is enough to compute the order of $(\cZ_{g,2}^{1,2})^q$. Thanks to Lemma~\ref{lem: Chow Z12 n=2}, this is divisible by $(4g+2)(g+1)$, as the restriction of $[\cZ_{g,2}^{1,2}]$ to $\cZ_{g,2}^{1,2}$ is equal to $-\psi_1$ by Lemma~\ref{lem:psiZ23}. Multiplying the relation in degree 1 with $[\cZ_{g,2}^{1,2}]$, we get that the order of $[\cZ_{g,2}^{1,2}]^q$ is exactly $(4g+2)(g+1)$. This concludes.
	\end{proof}
	\section{The almost computation of $\mathrm{CH}^*(\cH_{g,n})$ for $3\leq n\leq2g+3$}\label{sec: Hgn n>=3}
	\subsection{The Chow ring of $\cH_{g,n}^{\far}$ for $3\leq n\leq2g+3$}\label{subsec: Chow Hgnfar n>=3}
	In this Subsection we compute the Chow ring of $\cH_{g,n}^{\far}$ for $3\leq n\leq2g+2$, and obtain the generator of the ring in the case $n=2g+3$. In this last case, the only thing we do not manage to compute is the multiplicative order of the generator.
	
	Recall that
	\[
	\cH_{g,n}^{\far}\simeq(\cP(2^{2g+3-n},1^n)\times(\mathbb{A}^{n-3}\setminus\widetilde{\Delta}))\setminus\oD_{g,n}
	\]
	and
	\[
	\CH^*(\cP(2^{2g+3-n},1^n)\times(\mathbb{A}^{n-3}\setminus\widetilde{\Delta}))\simeq\frac{\mathbb{Z}[h]}{(2^{2g+3-n}h^{2g+3})}.
	\]
	\begin{proposition}\label{prop: Chow Hgnfar n>=3}
		Let $\charac k>2g$. If $3\leq n\leq 2g+2$, then
		\[
		\CH^*(\cH_{g,n}^{\far})\simeq\frac{\mathbb{Z}[[\cW_{g,n}^1]]}{((8g+4)[\cW_{g,n}^1],2[\cW_{g,n}^1]^2)}.
		\]
		Moreover, the pullback $\CH^*(\cH_{g,n-1}^{\far})\rightarrow\CH^*(\cH_{g,n}^{\far})$ along the map induced by forgetting the $n$-th section is an isomorphism for $4\leq n\leq2g+2$.
	\end{proposition}
	\begin{proof}
		First, we know that the above relations hold; see Lemma~\ref{lem:W1square} for the relation in degree 2. On the other hand, in~\cite{EH22} it has been shown that, for $n\leq2g+2$, the Chow ring of the intersection of all $\cW_{g,n}^i$ has Chow ring generated by the restriction of $\psi_1$, with ideal of relations generated by $2\psi_1$. Since for $3\leq n\leq2g+2$ the restrictions to $\cH_{g,n}^{\far}$ of the classes $\psi_1$ and $[\cW_{g,n}^1]$ agree, this concludes.
	\end{proof}
	\begin{remark}\label{rmk: Chow Hg2g+3far}
		The case $n=2g+3$ is harder, and we do not compute $\mathrm{CH}^*(\cH_{g,2g+3}^{\far})$. The difference with the case $n\leq2g+2$ is that $\cH_{g,2g+3}$ is a scheme, hence the integral Chow ring has to be 0 in degrees above the dimension. However, some things follow easily from what we have done. Recall that
		\[
		\cH_{g,2g+3}^{\far}\simeq(\mathbb{P}^{2g+2}\times(\mathbb{A}^{2g}\setminus\widetilde{\Delta}))\setminus\overline{\Delta}_{g,2g+3}
		\]
		and so the Chow ring is generated by $[\cW_{g,2g+3}^1]$, the relations of Proposition~\ref{prop: Chow Hgnfar n>=3} still hold, and $[\cW_{g,2g+3}^1]^{2g+2}=0$. It is not clear what is the exact multiplicative order of $[\cW_{g,2g+3}^1]$. Moreover, it is easy to convince ourselves that it is not possible to use similar Chow envelopes as in the cases $n=0,1,2$.
	\end{remark}
	\subsection{The Chow ring of $\cH_{g,3}$}
	We apply the same method as for $n=2$. Recall that we have an exact sequence
	\begin{equation*}
		\begin{tikzcd}
			\bigoplus_{i<j}\mathrm{CH}^*(\cZ_{g,3}^{i,j})\arrow[r] & \mathrm{CH}^*(\cH_{g,3})\arrow[r] & \mathrm{CH}^*(\cH_{g,3}^{\far})\arrow[r] & 0.
		\end{tikzcd}
	\end{equation*}
	\begin{lemma}\label{lem: Chow Z12 n=3}
		Let $\charac k>2g$. Then,
		\[
		\mathrm{CH}^*(\cZ_{g,3}^{1})\simeq\mathrm{CH}^*(\cH_{g,2}^{1,3})\simeq\frac{\mathbb{Z}[\psi_1]}{((4g+2)(g+1)\psi_1,(2g+1)(g+1)\psi_1^2)}.
		\]
	\end{lemma}
	\begin{proof}
		It follows immediately from the localization sequence and Theorem~\ref{thm:ChowHg2}.
	\end{proof}
	\begin{theorem}\label{thm: Chow Hg n=3}
		Let $\charac k>2g$. Then, $\mathrm{CH}^*(\cH_{g,3})$ is generated by the classes of $\cW_{g,3}^1$, $\cZ_{g,3}^{1,2}$, $\cZ_{g,3}^{1,3}$ and $\cZ_{g,3}^{2,3}$, and the ideal of relations is generated by
		\addtolength{\jot}{2pt}
		\begin{align*}
			\bullet & \ (8g+4)[\cW_{g,3}^1]-(4g+2)(g+1)([\cZ_{g,3}^{1,2}]+[\cZ_{g,3}^{1,3}]-[\cZ_{g,3}^{2,3}]),\\
			\bullet & \ [\cZ_{g,3}^{1,2}]\cdot[\cZ_{g,3}^{1,3}],\qquad[\cZ_{g,3}^{1,j}]\cdot[\cZ_{g,3}^{2,3}]\text{ for }j=2,3,\qquad
			[\cZ_{g,3}^{1,j}]\cdot[\cW_{g,3}^1]\text{ for }j=2,3,\\
			\bullet & \ [\cW_{g,3}^1]\cdot[\cZ_{g,3}^{2,3}]+(g+1)[\cZ_{g,3}^{2,3}]^2,\qquad2[\cW_{g,3}^1]^2-(g+1)[\cZ_{g,3}^{2,3}]^2,\\
			\bullet & \ (2g+1)(g+1)([\cZ_{g,3}^{1,2}]^2-[\cZ_{g,3}^{i,j}]^2)\text{ for }i\not=j,\qquad b_g[\cZ_{g,3}^{i,j}]^2\text{ for }i\not=j.
		\end{align*}
		where $b_g$ is either $(2g+1)(g+1)$ or $(4g+2)(g+1)$.
	\end{theorem}
	\begin{proof}
		We already know that those are the generators, and that the relations in degree 1 are multiples of the first relation in the statement. Moreover, we already know that the relations in the statement hold.
		
		Suppose there is another relation, which would take the form of
		\[
			a[\cW_{g,3}^1]^q+b[\cZ_{g,3}^{1,2}]^q+c[\cZ_{g,3}^{1,3}]^q+d[\cZ_{g,3}^{2,3}]^q.
		\]
		Restricting to $\cH_{g,3}^{\far}$ and applying Proposition~\ref{prop: Chow Hgnfar n>=3}, we get that we can assume $a=0$. Then, by restricting to $\cZ_{g,3}^{1,2}$ and using Lemma~\ref{lem: Chow Z12 n=3}, we reduce to the case $b=0$. Similarly, by restricting to $\cZ_{g,3}^{1,3}$ we get that we can assume $c=0$, hence we are left with computing the orders of $[\cZ_{g,3}^{i,j}]^q$ for $q\geq2$. Restricting to $\cZ_{g,3}^{i,j}$, Lemma~\ref{lem: Chow Z12 n=3} shows that their order is divisible by $(2g+1)(g+1)$. Notice that $(4g+2)(g+1)[\cZ_{g,3}^{i,j}]^q=0$ for all $q\geq3$, as the last line shows; therefore, the only missing data is the order $b_g$ of $[\cZ_{g,3}^{i,j}]^2$, which is $(2g+1)(g+1)$ or $(4g+2)(g+1)$.
	\end{proof}
	\subsection{The Chow ring of $\cH_{g,4}$}
	We start in the same way as before.
	\begin{lemma}
		Suppose $\charac k>2g$. Then, $\CH^*(\cZ_{g,4}^{1,2})$ is isomorphic to
		\[
		\mathrm{CH}^*(\cH_{g,3}^{1})\simeq\frac{\mathbb{Z}[\psi_1]}{((4g+2)(g+1)\psi_1,(g+1)\psi_1^2)}\simeq\frac{\mathbb{Z}[[\cZ_{g,3}^{2,3}]]}{((4g+2)(g+1)[\cZ_{g,3}^{2,3}],(g+1)[\cZ_{g,3}^{2,3}]^2)}.
		\]
	\end{lemma}
	\begin{proof}
		It follows immediately from Theorem~\ref{thm: Chow Hg n=3}.
	\end{proof}
	\begin{theorem}\label{thm: Chow Hgn n=4}
		Let $\charac k>2g$. Then, $\mathrm{CH}^*(\cH_{g,4})$ is generated by the classes of $\cW_{g,4}^1$, $\cZ_{g,4}^{1,2}$, $\cZ_{g,4}^{1,3}$, $\cZ_{g,4}^{1,4}$ and $\cZ_{g,4}^{2,3}$, and the ideal of relations is generated by
		\addtolength{\jot}{2pt}
		\begin{align*}
			\bullet & \ (8g+4)[\cW_{g,4}^1]-(4g+2)(g+1)([\cZ_{g,4}^{1,2}]+[\cZ_{g,4}^{1,3}]-[\cZ_{g,4}^{2,3}]),\\
			\bullet & \ [\cZ_{g,4}^{1,i}]\cdot[\cZ_{g,4}^{1,j}]\text{ for }1<i<j,\qquad[\cZ_{g,4}^{1,j}]\cdot[\cZ_{g,4}^{2,3}]\text{ for }j=2,3,\\
			\bullet & \ 
			[\cZ_{g,4}^{1,j}]\cdot[\cW_{g,3}^1]\text{ for }j>1,\qquad [\cZ_{g,4}^{2,3}]\cdot[\cW_{g,3}^1]+(g+1)[\cZ_{g,4}^{2,3}]^2,\\
			\bullet & \
			2[\cW_{g,4}^1]^2-(g+1)[\cZ_{g,4}^{2,3}]^2,\qquad[\cZ_{g,4}^{2,3}]^2+[\cZ_{g,4}^{2,3}]\cdot[\cZ_{g,4}^{1,4}],\qquad [\cZ_{g,4}^{2,3}]^2-[\cZ_{g,4}^{1,4}]^2,\\
			\bullet & \ (g+1)([\cZ_{g,4}^{1,2}]^2-[\cZ_{g,4}^{i,j}]^2)\text{ for }i\not=j,\qquad b_g[\cZ_{g,4}^{i,j}]^2\text{ for }i\not=j
		\end{align*}
		for some integer $b_g=b'_g(g+1)$ such that $b'_g$ divides $4g+2$.
	\end{theorem}
	\begin{proof}
		We already know that those are the generators and that the relations hold. The only relation which is less clear is the bottom left. However, recall that
		\[
		2[\cW_{g,4}^{1}]^2=(g+1)[\cZ_{g,4}^{2,3}]^2
		\]
		by Lemma~\ref{lem:W1Z23}. Permuting the indices, this gives
		\[
			(g+1)[\cZ_{g,4}^{2,3}]^2=(g+1)[\cZ_{g,4}^{2,4}]^2=(g+1)[\cZ_{g,4}^{1,3}]^2,
		\]
		where the last equality holds even without multiplying by $(g+1)$. Finally, the same argument as in Theorem~\ref{thm: Chow Hg n=3} applies to show that the expressions in the statement generate the ideal of relations.
	\end{proof}
	\subsection{The Chow ring of $\cH_{g,n}$ for $5\leq n\leq2g+3$} We adopt the same strategy as before.
	\begin{lemma}\label{lem:Chowringsintermediate5<=n<=2g+3}
		Let $\charac k>2g$, and let $5\leq n\leq2g+3$. Then,
		\[
		\mathrm{CH}^*(\cZ_{g,n}^{1,2})\simeq\mathrm{CH}^*(\cH_{g,n-1}^1)\simeq\frac{\mathbb{Z}[\psi_1]}{((4g+2)(g+1)\psi_1,\psi_1^2)}.
		\]
	\end{lemma}
	\begin{proof}
		It follows from Theorem~\ref{thm: Chow Hgn n=4} that $\psi_1$ is a generator satisfying $\psi_1^2=0$. As we also know its Picard group, this concludes.
	\end{proof}
	\begin{theorem}\label{thm: Chow Hgn 5<=n<=2g+2}
		Let $\charac k>2g$ and $5\leq n\leq2g+2$. Then, $\mathrm{CH}^*(\cH_{g,n})$ is generated by the classes of $\cW_{g,n}^1$, $\cZ_{g,n}^{1,i}$ for $1<i\leq n$ and $\cZ_{g,n}^{2,3}$, and the ideal of relations is generated by
		\addtolength{\jot}{2pt}
		\begin{align*}
			\bullet & \
			(8g+4)[\cW_{g,n}^1]-(4g+2)(g+1)([\cZ_{g,n}^{1,2}]+[\cZ_{g,n}^{1,3}]-[\cZ_{g,n}^{2,3}]),\\
			\bullet & \
			[\cZ_{g,n}^{1,i}]\cdot[\cZ_{g,n}^{1,j}]\text{ for }1<i<j,\qquad[\cZ_{g,n}^{1,j}]\cdot[\cZ_{g,n}^{2,3}]\text{ for }j=2,3,\\
			\bullet & \
			[\cZ_{g,n}^{1,j}]\cdot[\cW_{g,n}^1]\text{ for }j\not=1,\qquad[\cZ_{g,n}^{2,3}]\cdot[\cW_{g,n}^1]+(g+1)[\cZ_{g,n}^{2,3}]^2,\\
			\bullet & \
			2[\cW_{g,n}^1]^2-(g+1)[\cZ_{g,n}^{2,3}]^2,\qquad[\cZ_{g,n}^{2,3}]^2+[\cZ_{g,n}^{2,3}]\cdot[\cZ_{g,n}^{1,i}]\text{ for }i>3,\\
			\bullet & \
			[\cZ_{g,n}^{2,3}]^2-[\cZ_{g,n}^{1,j}]^2\text{ for }j\not=1,\qquad b_g[\cZ_{g,n}^{i,j}]^2\text{ for }i\not=j
		\end{align*}
		where $b_g$ is a positive integer dividing $(4g+2)(g+1)$.
	\end{theorem}
	\begin{proof}
		It can be proved in the same way as Theorems~\ref{thm: Chow Hg n=3} and~\ref{thm: Chow Hgn n=4}.
	\end{proof}
	\begin{remark}
		For $3\leq n\leq2g+2$, we are left with computing $b_g$, which we do not do in this article. There are some difficulties in using higher Chow rings (with $\mathbb{Z}_l$-coefficients) in this case, since it can be shown that the first higher Chow ring of $\cH_{g,n}^{\far}$ is non-zero in the cases we are interested in. Notice that $b_g$ is also the order of $[\cZ_{g,n}^{2,3}]\cdot[\cZ_{g,n}^{1,4}]$ for $n\geq4$, so that at least we are dealing with a proper intersection.
	\end{remark}
	The following Theorem can be proved in the same way.
	\begin{theorem}\label{thm: Chow Hgn n=2g+3}
		Let $\charac k>2g$ and $n=2g+3$. Then, $\mathrm{CH}^*(\cH_{g,2g+3})$ is generated by the classes of \/ $\cW_{g,2g+3}^1$, $\cZ_{g,2g+3}^{1,i}$ for $1<i\leq2g+3$ and $\cZ_{g,2g+3}^{2,3}$, and the ideal of relations is generated by
		\addtolength{\jot}{2pt}
		\begin{align*}
			\bullet & \
			(8g+4)[\cW_{g,2g+3}^1]-(4g+2)(g+1)([\cZ_{g,2g+3}^{1,2}]+[\cZ_{g,2g+3}^{1,3}]-[\cZ_{g,2g+3}^{2,3}]),\\
			\bullet & \
			[\cZ_{g,2g+3}^{1,i}]\cdot[\cZ_{g,2g+3}^{1,j}]\text{ for }1<i<j,\qquad[\cZ_{g,2g+3}^{1,j}]\cdot[\cZ_{g,2g+3}^{2,3}]\text{ for }j=2,3,\\
			\bullet & \
			[\cZ_{g,2g+3}^{1,j}]\cdot[\cW_{g,2g+3}^1]\text{ for }j\not=1,\qquad[\cZ_{g,2g+3}^{2,3}]\cdot[\cW_{g,2g+3}^1]+(g+1)[\cZ_{g,2g+3}^{2,3}]^2,\\
			\bullet & \
			2[\cW_{g,2g+3}^1]^2-(g+1)[\cZ_{g,2g+3}^{2,3}]^2,\qquad[\cZ_{g,2g+3}^{2,3}]^2+[\cZ_{g,2g+3}^{2,3}]\cdot[\cZ_{g,2g+3}^{1,i}]\text{ for }i>3,\\
			\bullet & \
			[\cZ_{g,2g+3}^{2,3}]^2-[\cZ_{g,2g+3}^{1,j}]^2\text{ for }j\not=1,\qquad b_g[\cZ_{g,2g+3}^{i,j}]^2\text{ for }i\not=j,\qquad[\cW_{g,2g+3}^1]^{c_g}.
		\end{align*}
		where $b_g$ is a positive integer dividing $(4g+2)(g+1)$, and $1<c_g\leq2g+2$.
	\end{theorem}
	
	\bibliographystyle{amsalpha}
	\bibliography{library}
	
\end{document}